\documentclass[12pt,reqno]{amsart} 

\usepackage{amssymb,bbm,enumitem,marginnote,url}
\usepackage{aliascnt}
\usepackage{doi}
\usepackage{pgfplots} 
\usepackage{esint}

\usepackage{psfrag}
\usepackage[margin=1.25in]{geometry}

\newcommand{\arxiv}[1]{%
 \href{http://front.math.ucdavis.edu/#1}{ArXiv:#1}}
\usepackage{color}

\newcommand\blue[1]{\textcolor{blue}{#1}}

\newtheorem{theorem}{Theorem}%[section]

\newaliascnt{lemma}{theorem}
\newtheorem{lemma}[lemma]{Lemma}
\aliascntresetthe{lemma}

\newaliascnt{proposition}{theorem}
\newtheorem{proposition}[proposition]{Proposition}
\aliascntresetthe{proposition}

\newaliascnt{corollary}{theorem}
\newtheorem{corollary}[corollary]{Corollary}
\aliascntresetthe{corollary}

\newaliascnt{conjecture}{theorem}

\aliascntresetthe{conjecture}

\newaliascnt{openQ}{theorem}

\aliascntresetthe{openQ}

\newaliascnt{quest}{theorem}

\aliascntresetthe{quest}

\newaliascnt{questx}{conjx}

\aliascntresetthe{questx}

\theoremstyle{definition}

\newaliascnt{defn}{theorem}

\aliascntresetthe{defn}

\newaliascnt{example}{theorem}

\aliascntresetthe{example}

\newaliascnt{rem}{theorem}

\aliascntresetthe{rem}

\makeatletter
\def\tagform@#1{\maketag@@@{\ignorespaces#1\unskip\@@italiccorr}}
\let\orgtheequation\theequation
\def\theequation{(\orgtheequation)}
\makeatother
\def\equationautorefname~{}
\newtheorem*{rem*}{Remarks}
\newtheorem*{example*}{Example}

\newcommand{\arctanh}{\operatorname{arctanh}}

\newcommand{\sign}{\operatorname{sign}}

\newcommand{\D}{\mathbb{D}}

\newcommand{\e}{\varepsilon}
\newcommand{\B}{{\mathbb B}}
\newcommand{\Bn}{{{\mathbb B}^n}}
\newcommand{\R}{{\mathbb R}}
\newcommand{\Rn}{{{\mathbb R}^n}}

\let\oldmarginnote\marginnote
\renewcommand{\marginnote}[1]{\oldmarginnote{\tiny \blue{#1}}}

\begin{document}
\title[Center of mass --- hyperbolic well-posedness]{Well-posedness of Hersch--Szeg\H{o}'s center of mass by hyperbolic energy minimization}


	\begin{abstract}
The hyperbolic center of mass of a finite measure on the unit ball with respect to a radially increasing weight is shown to exist, be unique, and depend continuously on the measure. Prior results of this type are extended by characterizing the center of mass as the minimum point of an energy functional that is strictly convex along hyperbolic geodesics. A special case is Hersch's center of mass lemma on the sphere, which follows from convexity of a logarithmic kernel introduced by Douady and Earle.   
	\end{abstract}
	
\author[]{R. S. Laugesen}
\address{Department of Mathematics, University of Illinois, Urbana,
	IL 61801, U.S.A.}
\email{Laugesen@illinois.edu}

	\maketitle	

\medskip

\section{\bf Introduction} \label{sec:intro}

\subsection*{Motivation} 
The hyperbolic center of mass of a finite measure $\mu$ on the closed unit ball is the point $c$ for which $\int T_{-c}(y) \, d\mu(y) = 0$, where the M\"{o}bius transformation $T_x$ gives hyperbolic translation by $x$ on the ball. Equivalently, the pushforward measure has its center of mass at the origin: $\int y \, d[\big(T_{-c} \big)_{\! *} \, \mu] = 0$. 

This paper establishes well-posedness for generalized centers of mass involving radial weights, which arise in the proofs of sharp upper bounds on eigenvalues of the Laplacian in hyperbolic space and the sphere. These generalized centers of mass will be shown to exist, be unique, and depend continuously on the measure. 

Consider a radial weight $g(r)$ with $g(0)=0$, as illustrated in \autoref{fig:ggraph}. The task is to find conditions on $g$ and the measure $\mu$ on the closed ball $\overline{\Bn}$ under which the generalized hyperbolic center of mass equation 
\begin{equation} \label{eq:gcenter}
\int_{\overline{\Bn}} g(|T_x(y)|) \frac{T_x(y)}{|T_x(y)|} \, d\mu(y) = 0 
\end{equation}
has a solution $x \in \Bn$, and to determine when this point $x$ is unique and depends continuously on $\mu$. In the special case $g(r)=r$, condition \eqref{eq:gcenter} reduces to the original center of mass equation $\int_{\overline{\Bn}} T_x(y) \, d\mu(y) = 0$, in which case $x=-c$. 
\begin{figure}
\begin{center}
\includegraphics[scale=0.35]{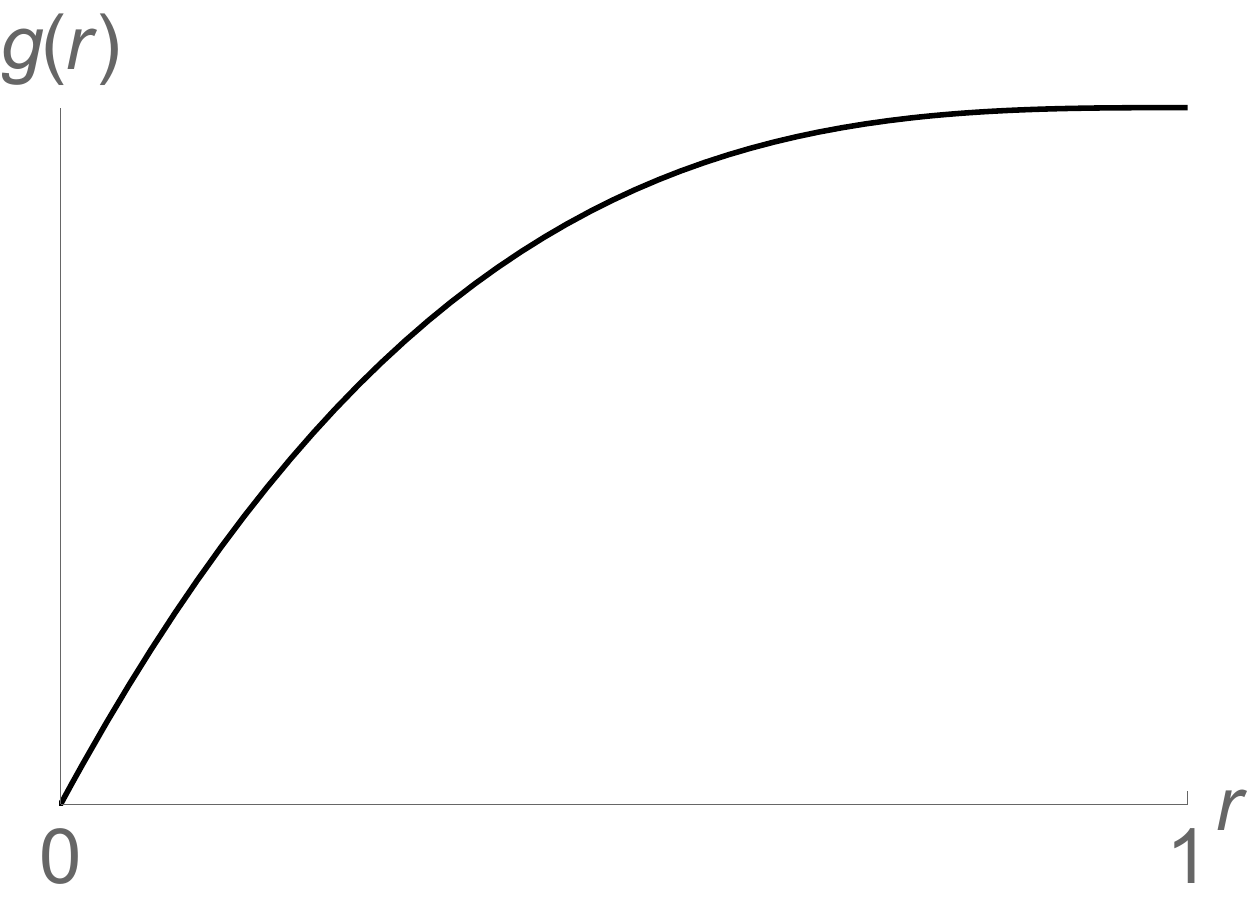} \qquad \includegraphics[scale=0.35]{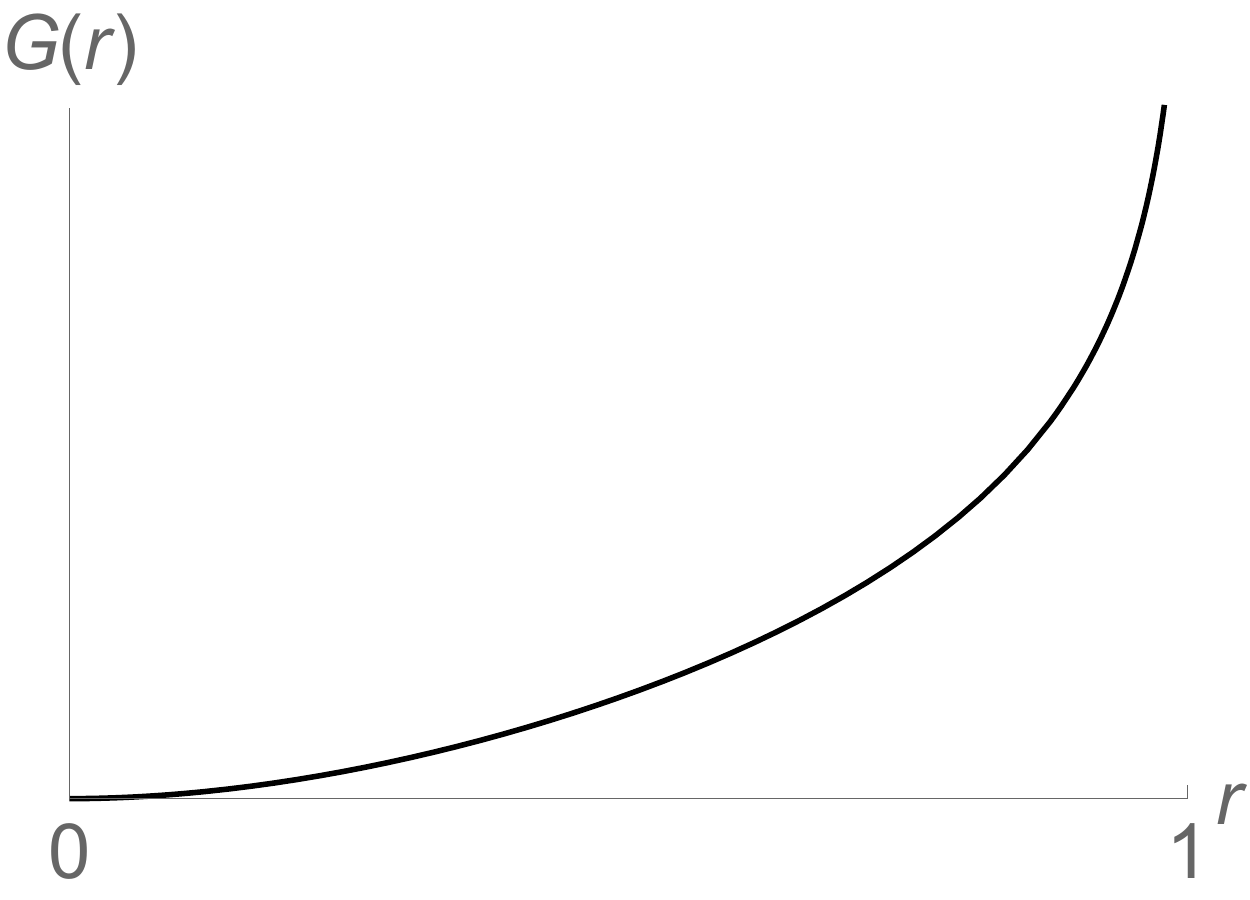}
\end{center}
\caption{\label{fig:ggraph}\textsc{Left:} Example of a radial weight $g(r)$, with $g(0)=0$. The existence results in this paper do not assume $g$ to be nonnegative or increasing. The uniqueness and continuous dependence results assume $g$ is positive and increasing. \textsc{Right:} The energy kernel $G$ is the hyperbolic antiderivative of $g$, and so $G$ is hyperbolically convex if $(1-r^2)G^\prime(r)=g(r)$ is increasing, as in the example shown.}
\end{figure}

Condition \eqref{eq:gcenter} may be expressed intrinsically in terms of hyperbolic distances and the exponential map, as  described in \autoref{sec:diffgeometric}. 

The existence theorems in this paper, which show equation \eqref{eq:gcenter} has a solution $x$, are motivated by work of Szeg\H{o} \cite[p.~351]{S54} in the open disk and Hersch \cite[p.\,1645]{H70} on the sphere. Szeg\H{o} needed to normalize the $g$-center of mass by conformal mapping, so that he could construct valid trial functions for his proof that the disk maximizes the second Neumann eigenvalue of the Laplacian among all simply connected planar domains of given area. Hersch similarly needed to move the center of mass to the origin for measures supported on the sphere, in order to show the round sphere maximizes the second eigenvalue of the Laplace--Beltrami operator, among metrics of given area. 

Underlying the Szeg\H{o} and Hersch existence proofs is the Brouwer fixed point theorem, or equivalent tools from topological index theory. The key to that existence result is that when $|x|=1$ the M\"{o}bius transformation is essentially constant, satisfying $T_x(y)=x$ for every $y$ except $y=-x$, and so the left side of \eqref{eq:gcenter} equals a multiple of $x$. In particular, that left side vector field points outward on the boundary of the ball, and hence must vanish somewhere inside the ball, giving a solution of \eqref{eq:gcenter}.

The uniqueness and continuous dependence results in the paper are motivated by work of Girouard, Nadirashvili and Polterovich. They proved well-posedness of the $g$-center of mass for measures on the $2$-dimensional disk \cite[Lemmas 2.2.3--2.2.5, 3.1.1]{GNP09}, \cite[Proposition 3.1]{GP10}, and also for measures on the sphere in all dimensions \cite[Proposition 4.1.5]{GNP09}. 

The current paper establishes well-posedness of the $g$-center of mass for measures on balls in all dimensions. The methods are analytic rather than topological in nature, relying on minimization of an explicitly defined energy functional. 

\subsection*{Overview of results} 

\autoref{th:hyperbolic} proves well-posedness of the $g$-center of mass for compactly supported measures in the open ball, assuming for existence that $\int_0^1 g(r) (1-r^2)^{-1} \, dr = \infty$, and assuming for uniqueness and continuous dependence that $g$ is strictly increasing, or else that $g$ is merely increasing and $\mu$ is not supported in a hyperbolic geodesic. 

\autoref{cor:weinhyp} deduces a Weinberger-type center of mass result for densities. This result was used for maximizing the second Neumann eigenvalue among bounded domains of given volume in hyperbolic space, by Chavel \cite[p.\ 80] {C80}; see also Ashbaugh and Benguria \cite[{\S}6]{AB95}. In those works, $g(r)$ is increasing out to a certain radius and then constant for all larger values of $r$, which is why it matters in this paper to treat weight functions $g$ that are non-strictly increasing. 

A ``folded'' variant in \autoref{cor:weinfoldhyp} recovers a result of Girouard, Nadirashvili and Polterovich \cite{GNP09}.

Measures on the closed ball are treated in \autoref{th:herschszego}, getting well-posedness of the center of mass when $g$ is increasing and $\mu$ is not supported in a hyperbolic geodesic. Point masses on the sphere are permitted, provided each point contributes less than half the total mass of the ball. 

Hersch's center of mass lemma for measures on the sphere is deduced in \autoref{cor:spheremu}, with a version for densities in \autoref{cor:hersch}. Measures on the sphere should be regarded as living on the boundary at infinity of the hyperbolic ball. 

The Szeg\H{o} situation involving simply connected domains in the plane is recovered in \autoref{cor:szego}, and Weinstock's analogous result for measures on planar Jordan curves \cite[pp.~748--749]{W54}, which he needed for estimating the first positive Steklov eigenvalue, appears in \autoref{cor:weinstock}. 

If uniqueness and continuous dependence are not needed and one aims merely for the existence of a center of mass point,  then as shown in \autoref{th:herschszegosigned}, one may handle signed measures. 

\subsection*{Summary of the energy method}  To find the center of mass in euclidean space one minimizes the moment of inertia $\int \frac{1}{2} |y-c|^2 \, d\mu(y)$ with respect to the choice of center point $c$. The analogous quantity to minimize for the  $g$-center of mass on the hyperbolic ball is the \emph{energy functional}
\[
E(x) = \int_\Bn G(|T_x(y)|) \, d\mu(y) , \qquad x \in \Bn , 
\]
where $G^\prime(r) = g(r)/(1-r^2)$. Clearly this energy is finite if the measure $\mu$ has compact support in the open ball. The gradient of the energy is the vector field on the left side of \eqref{eq:gcenter} (up to a factor; see formula \eqref{eq:hyperbolicgradient} later), and so critical points of the energy, in particular minimum points of the energy, are automatically centers of mass. 

To prove \autoref{th:hyperbolic}, for existence of an energy minimizing point we show the energy tends to infinity as $|x| \to 1$, while for uniqueness and continuous dependence we prove the energy is strictly hyperbolically convex.

This energy method can break down if the support of the measure extends out to the boundary sphere. Indeed, in that case $E(x)$ can equal $+\infty$ at every point. Such singularities will be avoided in \autoref{th:herschszego} for measures on the closed ball by renormalizing the energy: let
\[
\mathcal{E}(x) = \int_{\overline{\Bn}} K(x,y) \, d\mu(y) , \qquad x \in \Bn , 
\]
where the renormalized (or relative) kernel $K(x,y)$ is the continuous extension of $G(|T_x(y)|) - G(|y|)$ to the boundary sphere with respect to the $y$-variable. \autoref{sec:renormalizedenergy} develops the properties of this renormalized, extended kernel. 

Incidentally, the energy minimization approach in this paper suggests that the hyperbolic center of mass could be computed efficiently by a steepest descent or Newton algorithm. Such numerical methods would be particularly efficient when $g$ is  increasing, since then the energy is hyperbolically convex. Such gradient descent methods have been investigated by Afsari, Tron and Vidal \cite{ATV13} for $L^p$-Riemannian centers of mass (which are mentioned in \autoref{sec:diffgeometric} below). In contrast, the index theory approach to proving existence of a center of mass does not suggest any practical method for finding it. 

\subsection*{Energy method for Hersch's center of mass on the sphere} In the special case where the measure $\mu$ is supported entirely on the unit sphere (\autoref{cor:spheremu}), the energy method for proving Hersch's center of mass normalization is due to Douady and Earle \cite[Sections 2,11]{DE86} and Millson and Zombro \cite[Section 4]{MZ96}. Douady and Earle used the energy method for uniqueness, having already proved existence by index theory. Millson and Zombro showed how to get both existence and uniqueness from properties of the energy, yielding the following results, which are justified in their paper and in \autoref{sec:herschszegoproof} below.

Consider a Borel measure $\mu$ on $S^{n-1}, n \geq 2$, that satisfies $0 < \mu(S^{n-1}) < \infty$ and the point mass condition $\mu( \{ y \} ) < \frac{1}{2} \mu(S^{n-1})$ for all $y \in S^{n-1}$. The renormalized energy can be written explicitly in this situation as 
\begin{equation} \label{eq:introrenormalized}
\mathcal{E}_{\text{sphere}}(x) = \int_{S^{n-1}} \frac{1}{2} \log \frac{|x+y|^2}{1-|x|^2} \, d\mu(y) , \qquad x \in \Bn . 
\end{equation}
This energy is strictly hyperbolically convex, and it tends to infinity as $|x| \to 1$. Hence it possesses a unique minimizing point $x=x(\mu)$. The gradient vanishes at this critical point, which yields the hyperbolic center of mass equation $\int_{S^{n-1}} T_x(y) \, d\mu(y) = 0$. 

The logarithmic kernel in \eqref{eq:introrenormalized} is exactly the Busemann function for the boundary at infinity of the hyperbolic ball \cite[p.\,27]{DE86}, \cite[Section 4]{MZ96}, and indeed the kernel will be derived that way in formula \eqref{eq:distrelation}. 

\subsection*{Related literature for euclidean space, the sphere, Riemannian manifolds}

The center of mass results in this paper for the hyperbolic ball have analogues in euclidean space, as explored in my recent paper \cite{L20e}. 

The earliest continuous dependence result I know for Hersch's center of mass is due to Chang and Yang \cite[Appendix]{CY87}, in their work on prescribing the curvature of a metric on the sphere. Morpurgo \cite[p.\,362]{M96} later applied their approach in proving local minimality of the round sphere for the heat trace. 

Hersch's result continues to play a role in new applications, such as by Branson, Fontana and Morpurgo \cite{BFM13} for sharp Moser--Trudinger and Beckner--Onofri inequalities on the CR sphere, and by Frank and Lieb \cite{FL12a,FL12b} for the sharp Hardy--Littlewood--Sobolev inequality in euclidean space and the Heisenberg group.

The Riemannian center of mass on a nonpositively curved manifold arises from energy minimization with kernel $d(x,y)^p/p$. When specialized to the hyperbolic ball, these results give existence and uniqueness of the center of mass in our \autoref{th:hyperbolic} for the choice $g(r)=(\arctanh r)^{p-1}$. \autoref{sec:diffgeometric} explains the connection.

\section{\bf Notation and M\"{o}bius isometries} 

Write $\Bn$ for the open unit ball centered at the origin in $\Rn, n \geq 1$. Put
\[
s = s(r) = \frac{1}{2} \log \frac{1+r}{1-r} = \arctanh r , \qquad -1<r<1 ,
\] 
so that $ds=(1-r^2)^{-1} \, dr$ is the hyperbolic arclength element in the radial direction. 
Let $d_\Bn(x,y)$ be the hyperbolic distance between points $x$ and $y$ in the ball. In particular,  
\[
d_\Bn(x,0) = s(|x|) = \frac{1}{2} \log \frac{1+|x|}{1-|x|} 
\]
is the hyperbolic distance from $x$ to the origin. 
%
%\[
%\sinh d_\Bn(x,y) = \frac{|x-y|}{\sqrt{(1-|x|^2)(1-|y|^2)}}
%\]
%
Recall that hyperbolic geodesics in the unit ball are either straight lines through the origin, or arcs of circles that meet the unit sphere at right angles. 

The center of mass condition involves a family of M\"{o}bius transformations 
\[
T_x : \overline{\Bn} \to \overline{\Bn}
\]
that are parameterized by $x \in \Bn$ and have the following properties: $T_0(y)=y$ is the identity, and when $x \neq 0$ the map $T_x(\cdot)$ is a M\"{o}bius self-map of the ball such that $T_x(0)=x$ and $T_x$ fixes the points $\pm x/|x|$ on the unit sphere. 

In $1$ dimension the maps are
\[
T_x(y) = \frac{x+y}{1+xy} , \qquad x \in (-1,1), \ y \in [-1,1] ,
\]
so that 
\begin{equation} \label{eq:1dimtanh}
T_{\tanh a}(\tanh b)=\tanh (a+b) , \qquad a,b \in \R .
\end{equation}
That is, $T_{\tanh a}$ acts as translation by $a$, with respect to hyperbolic arclength on the interval $(-1,1)$. 
In $2$ dimensions the maps can be written in complex notation as 
\[
T_x(y) = \frac{x+y}{1+\overline{x}y} , \qquad x \in \D , \ y \in \overline{\D} , % \label{eq:Txy}
\]
where $\D \simeq \B^2$ is the unit disk in the complex plane. 
%
%To construct the maps in all dimensions, we choose $T_0(y)=y$ to be the identity map when $x=0$, and when $x \in \Bn \setminus \{ 0 \}$ let 
%\[
%T_x(y) = (I_x \circ R_x)(y) , \qquad y \in \overline{\Bn} ,
%\]
%where $R_x$ is reflection in the plane $\{  y : x \cdot y = 0 \}$ perpendicular to $x$ and  $I_x$ is inversion in the sphere of radius $\sqrt{|x|^{-2}-1}$ centered at the point $x/|x|^2$. The formula for this transformation is more complicated than in $1$ and $2$ dimensions, but still manageable:
%\[
%T_x(y) 
%= \frac{x}{|x|^2} + (|x|^{-2}-1) \frac{y - (2y \cdot x + 1) x/|x|^2}{\big| y - (2y \cdot x + 1) x/|x|^2 \big|^2} 
%\]
%and so after simplifying, 
%
In all dimensions \cite[eq.\,(26)]{A81}:
\begin{equation}
T_x(y) = \frac{(1+2x \cdot y+|y|^2)x +(1-|x|^2)y}{1+2x \cdot y + |x|^2|y|^2} , \qquad x \in \Bn, \ y \in \overline{\Bn} .\label{eq:Mobius}
\end{equation}
Observe $T_x(y)$ is a continuous function mapping $(x,y) \in \Bn \times \overline{\Bn}$ to $T_x(y) \in \overline{\Bn}$, and $T_x(\cdot)$ maps $\Bn$ to itself and $\partial \Bn$ to itself, with $T_x(0)=x$ and inverse $(T_x)^{-1} = T_{-x}$. Each $T_x$ is a hyperbolic isometry \cite[Section 2.7]{A81}.

\section{\bf Well-posedness results on the open ball}  \label{sec:resultshyperbolic}

Assume throughout this section that 
\[
\text{$g(r)$ is continuous and real valued for $0 \leq r < 1$, with $g(0)=0$,}
\]
and $\mu$ is a Borel measure on the open unit ball $\Bn, n \geq 1$, with
\[
0 < \mu(\Bn) < \infty .
\]
A typical radial profile $g$ is shown in \autoref{fig:ggraph}, although not all our results will assume $g$ is nonnegative and increasing like in the figure. 

Define $v : \Bn \to \Rn$ to be the radial vector field with magnitude $g$, meaning
\[
v(y) = g(|y|) \frac{y}{|y|} , \qquad y \in \Bn \setminus \{ 0 \} ,
\]
and $v(0)=0$. In other words, $v(r\hat{y})=g(r)\hat{y}$ whenever $0 \leq r < 1$ and $\hat{y}$ is a unit vector. Notice $v$ is continuous at the origin, since $g(0)=0$. 

Define a vector field $V : \Bn \to \Rn$ by integrating over M\"{o}bius translates of $v$:
\[
V(x) = \int_{\Bn} v(T_x(y)) \, d\mu(y) , \qquad x \in \Bn .
\]
This $V$ is well defined if the finite measure $\mu$ has compact support in $\Bn$. We seek a point $x_c$ for which $V(x_c)=0$, because then $x_c$ satisfies \eqref{eq:gcenter}, and so the antipodal point $-x_c$ is a $g$-center of mass for $\mu$. 
\begin{theorem}[Center of mass for compactly supported measures] \label{th:hyperbolic} \ \\
Assume the Borel measure $\mu$ has compact support in $\Bn$, with $0 < \mu(\Bn) < \infty$.

\noindent (a) [Existence] If $\int_0^1 g(r) (1-r^2)^{-1} \, dr = \infty$ then $V(x_c)=0$ for some $x_c \in \Bn$. 

\noindent (b) [Uniqueness] If either 
\begin{enumerate}[label=(\roman*),nosep]
\item $g$ is strictly increasing, or
\item $g$ is increasing, $g(r)>0$ whenever $0<r<1$, and $\mu$ is not supported in a hyperbolic geodesic, 
\end{enumerate}
then the point $x_c$ is unique.  

\noindent (c) [Continuous dependence] Suppose $\mu_k \to \mu$ weakly, where the $\mu_k$ are Borel measures all supported in a fixed compact subset of $\Bn$ and satisfying $0 < \mu_k(\Bn) < \infty$. If either (i) holds or else (ii) holds for $\mu$ and each $\mu_k$, then $x_c(\mu_k) \to x_c(\mu)$ as $k \to \infty$. 
\end{theorem}
The theorem is proved in \autoref{sec:hyperbolicproof}. For $g$ strictly increasing and bounded in $2$ dimensions, the theorem is due to Girouard, Nadirashvili and Polterovich \cite[Lemmas 2.2.3--2.2.5, 3.1.1]{GNP09}, \cite[Proposition 3.1]{GP10}. Their measures were permitted to take support in the whole closed disk, provided the boundary circle has no point masses. For more about closed disks and balls, see \autoref{th:herschszego} below. Girouard, Nadirashvili and Polterovich relied on topological methods (winding numbers) to prove existence, and obtained uniqueness through some ingenious estimates. See also the Riemannian center of mass results in \autoref{sec:diffgeometric}, for work of Grove, Karcher, Afsari and others.

\begin{rem*}
1. The integral condition in part (a) means $\int_0^\rho g(r)(1-r^2)^{-1} \, dr \to \infty$ as $\rho \to 1$. This hypothesis certainly holds if $g(1)>0$, but also holds for some functions that vanish at $r=1$, such as $g(r) = r/[\log 2/(1-r)]$.

2. The hypothesis that $\mu$ not be supported in a hyperbolic geodesic, in part (b)(ii), means $\mu(\Bn \setminus \gamma)>0$ for every hyperbolic geodesic $\gamma$ in the unit ball.

3. Uniqueness can fail in part (b)(ii) when the measure $\mu$ is supported in a hyperbolic geodesic, as the following example shows already in $1$ dimension. Take $g(r) = \min(r,1/2)$, so that $g$ increases from $0$ to $1/2$ for $r \in [0,1/2]$ and is constantly $1/2$ for $r \in [1/2,1)$, and suppose $\mu=\delta_a+\delta_b$ is a sum of point masses at locations $a$ and $b$ with $a \in (-1,-1/2)$ and $b \in (1/2,1)$. Then whenever $x$ is close enough to $0$ that $T_x(a) < -1/2$ and $T_x(b) > 1/2$, one has 
\[
V(x) = v(T_x(a)) + v(T_x(b)) = g(|T_x(a)|)  \cdot (-1) + g(|T_x(b)|) \cdot 1 = -1/2+1/2=0 .
\]
Thus $V$ vanishes for a whole interval of $x$ values close to $0$, destroying any hope of uniqueness. 

4. Uniqueness can also fail in \autoref{th:hyperbolic}(b) when $g$ is not increasing. For example, let $g(r)=\min(s(r),1/s(r))$, so that $g$ first increases and then decreases. In dimension $n=1$, choose $\mu=\delta_{\tanh 2} + \delta_{-\tanh 2}$ to be a sum of point masses at $\pm \tanh 2$. For $a \in [0,2]$ one finds with the help of the hyperbolic translation formula \eqref{eq:1dimtanh} that $V(\tanh a) = g(\tanh(2+a))-g(\tanh(2-a))$. Hence $V(0)=0,V(\tanh 1)=g(\tanh 3)-g(\tanh 1)=-2/3, V(\tanh 2)=g(\tanh 4)-g(0)=1/4$, and so $V(x)=0$ at $x=0$ and also at some $x$ between $\tanh 1$ and $\tanh 2$. Thus $V$ vanishes at more than one point. This counterexample extends immediately to higher dimensions, and there the measure can be smeared out symmetrically so it is not supported in a line.  

5. Continuous dependence can fail in part (c) when the measures are not all contained in a compact subset of $\Bn$. For example, in $1$ dimension consider $g(r)=s(r)$ and the measure $\mu_k=(1-1/k)\delta_0+(1/k)\delta_{\tanh (k^2)}$. Then $x_c(\mu_k)=-\tanh k$ since 
\[
V(-\tanh k) =  (1-1/k) s(\tanh k) (-1) + (1/k) s(\tanh(k^2-k)) (+1) = 0 ,
\]
where we used the hyperbolic translation formula \eqref{eq:1dimtanh}. Hence $x_c(\mu_k) \to -1$ as $k \to \infty$, even though $\mu_k$ converges weakly to $\mu=\delta_0$, which has $x_c(\mu)=0$.

6. The ``fixed compact support'' assumption in part (c) can be dropped if $g$ is continuous up to $r=1$, by \autoref{th:herschszego}(c) below. 
\end{rem*}

For the next corollary, recall the volume factor on the hyperbolic ball is $(1-|y|^2)^{-n}$. 
\begin{corollary}[Weinberger type orthogonality for a hyperbolic domain] \label{cor:weinhyp}
Suppose $\Omega$ is an open set with compact closure in $\Bn$ and $f$ is nonnegative on $\Omega$ with $0 < \int_\Omega f(y) (1-|y|^2)^{-n} dy < \infty$. If $\int_0^1 g(r) (1-r^2)^{-1} \, dr = \infty$ then a point $x \in \Bn$ exists such that each component of the vector field $v \circ T_x$ is orthogonal to $f$ with respect to the hyperbolic metric, meaning
\[
\int_\Omega v(T_x(y)) f(y) (1-|y|^2)^{-n} dy = 0 .
\]
If in addition $g$ is increasing with $g(r)>0$ whenever $0<r<1$ then the point $x$ is unique.  
\end{corollary}
\begin{proof} 
Apply \autoref{th:hyperbolic} parts (a) and (b)(ii) with $d\mu(y) = f(y) (1-|y|^2)^{-n} dy \big|_\Omega$. This measure $\mu$ equals a density times Lebesgue measure on $\Omega$, and so is not supported in any hyperbolic geodesic.
\end{proof}
The existence statement in the corollary is a hyperbolic analogue of a euclidean result by Weinberger \cite{W56}. It was mentioned in passing by Chavel \cite{C80} and Ashbaugh and Benguria \cite[{\S}6]{AB95}. The first detailed proof of existence seems to have been presented later by Benguria and Linde \cite[Theorem 6.1]{BL07}. Both Chavel and Ashbaugh--Benguria needed the case $f \equiv 1$, as part of a proof that the ball maximizes the second Neumann eigenvalue among bounded domains of given hyperbolic volume. Benguria and Linde pursued an analogous PPW-type result for the second Dirichlet eigenvalue, for which they needed nonconstant $f$. 

A well-posedness result involving hyperbolic folds can be developed as follows. Let 
\[
H_p = \{ y \in \Bn :  y \cdot p \leq 0 \} , \qquad p \in S^{n-1} ,
\]
be the closed hyperbolic halfball with normal vector $p$, whose boundary relative to the hyperbolic ball is $\partial H_p = \{ y \in \Bn :  y \cdot p = 0 \}$. Define
\[
H = H(p,t) = T_{pt}(H_p) , \qquad p \in S^{n-1}, \quad t \in (-1,1) ,
\]
be the image of that halfball under the M\"{o}bius translation $T_{pt}$. The boundary relative to the hyperbolic ball is $\partial H(p,t) = T_{pt}(\partial H_p)$. After writing
\[
R_p(y) = y - 2(y \cdot p) p
\]
for the reflection map across $\partial H_p$, the hyperbolic reflection across $\partial H(p,t)$ is defined by conjugation as
\[
R=R_{p,t}(y) = T_{pt} \circ R_p \circ (T_{pt})^{-1} .
\]
Define the ``fold map'' onto $H$ by 
\[
F(y) = 
\begin{cases}
y & \text{if\ } y \in H, \\
R(y) & \text{if\ } y \in \Bn \setminus H ,
\end{cases}
\]
so that the fold map fixed each point in $H$ and maps each point in $\Bn \setminus H$ to its hyperbolic reflection across $\partial H$.
\begin{corollary}[Orthogonality with a hyperbolic fold] \label{cor:weinfoldhyp}
Suppose $\Omega \Subset \Bn$ and $f$ is nonnegative on $\Omega$ with $0 < \int_\Omega f(y) (1-|y|^2)^{-n} dy < \infty$. If $\int_0^1 g(r)(1-r^2)^{-1} \, dr = \infty$, and $H$ and its fold map $F$ are given, then a point $x \in \Bn$ exists such that each component of the vector field $v \circ T_x \circ F$ is orthogonal to $f$, meaning
\[
\int_\Omega v(T_x \circ F(y)) f(y) (1-|y|^2)^{-n} dy = 0 .
\]
If in addition $g$ is increasing with $g(r)>0$ whenever $0<r<1$ then the point $x=x(H)$ is unique and depends continuously on the parameters $(p,t)$ of $H$.  
\end{corollary}
The proof is in \autoref{sec:weinfoldhypproof}. A ``folded'' corollary of this nature was obtained in $2$ dimensions by Girouard, Nadirashvili and Polterovich \cite[{\S}2.5]{GNP09} for maximizing the third Neumann eigenvalue (the second positive eigenvalue) of simply connected domains. Their construction was employed again by Girouard and Laugesen \cite[Lemma 9]{GL19} when maximizing the third Robin eigenvalue. More precisely, all these authors used not \autoref{cor:weinfoldhyp} but rather an analogous corollary on the whole closed disk that can be deduced from \autoref{th:herschszego}. 

\section{\bf Well-posedness results on the closed ball}  \label{sec:resultsclosedball}

Assume throughout this section that $g$ extends continuously to $r=1$: 
\[
\text{$g(r)$ is continuous and real valued for $0 \leq r \leq 1$, with $g(0)=0$,}
\]
and that $\mu$ is a Borel measure on the closed unit ball $\overline{\Bn}, n \geq 1$, with
\[
0 < \mu(\overline{\Bn}) < \infty .
\]
The vector field $v(r\hat{y})=g(r)\hat{y}$ extends to the closed ball ($0 \leq r \leq 1$), and so $V$ can be defined by integration with respect to $\mu$ over the closed ball:
\[
V(x) = \int_{\overline{\Bn}} v(T_x(y)) \, d\mu(y) , \qquad x \in \Bn .
\]
Again we seek a point $x_c$ at which $V(x_c)=0$. 

By the ``closure'' of a hyperbolic geodesic we mean its euclidean closure, consisting of the geodesic together with its endpoints on the unit sphere.  
\begin{theorem}[Center of mass on the closed ball] \label{th:herschszego} Assume $\mu$ is a Borel measure satisfying $0 < \mu(\overline{\Bn}) < \infty$ and 
\begin{equation} \label{eq:pointmass}
\mu( \{ y \} ) < \frac{1}{2} \mu(\overline{\Bn}) , \qquad y \in \partial \Bn .
\end{equation}
\noindent (a) [Existence] If $g(1)>0$ then $V(x_c)=0$ for some $x_c \in \Bn$. 

\noindent (b) [Uniqueness] If either 
\begin{enumerate}[label=(\roman*),nosep]
\item $g$ is strictly increasing, or
\item $g$ is increasing, $g(r)>0$ whenever $0<r<1$, and $\mu$ is not supported in the closure of a hyperbolic geodesic, 
\end{enumerate}
then the point $x_c$ is unique.  

\noindent (c) [Continuous dependence] Suppose $\mu_k \to \mu$ weakly, where the $\mu_k$ are Borel measures satisfying \eqref{eq:pointmass} and $0 < \mu_k(\overline{\Bn}) < \infty$ for all $k$. If either (i) holds or else (ii) holds for $\mu$ and each $\mu_k$, then $x_c(\mu_k) \to x_c(\mu)$ as $k \to \infty$. 
\end{theorem}
The proof of the theorem is in \autoref{sec:herschszegoproof}. In $2$ dimensions, when $g$ is strictly increasing and $\mu$ has no point masses on the unit circle, \autoref{th:herschszego} is due to Girouard, Nadirashvili and Polterovich \cite[Lemmas 2.2.3--2.2.5 and 3.1.1]{GNP09}, \cite[Proposition 3.1]{GP10}. Aubry, Bertrand and Colbois \cite[Lemma 4.11]{ABC09} proved existence of the center of mass for densities in the open hyperbolic ball, not necessarily compactly supported. That result is covered by \autoref{th:herschszego}(a). 

The point mass hypothesis \eqref{eq:pointmass} says that each point on the boundary possesses less than half the total mass of the ball. This hypothesis is essentially necessary for $V$ to have a vanishing point, assuming $g$ achieves its maximum at $r=1$ (e.g., if $g$ is increasing), as we now explain. Let $\hat{y} \in \partial \Bn$, so that $z=T_x(\hat{y})$ is a unit vector too. If $V(x)=0$ then 
\[
0 = V(x) \cdot z = \int_{\overline{\Bn} \setminus \{ \hat{y} \}} v(T_x(y)) \cdot z \, d\mu(y) + g(1) \mu(\{ \hat{y} \}) z \cdot z ,
\]
and of course $z \cdot z = 1$. Solving for the mass at $\hat{y}$, and then using that $g$ is largest at $r=1$, we find
\begin{align*}
\mu(\{ \hat{y} \}) 
& = \int_{\overline{\Bn} \setminus \{ \hat{y} \}} \frac{g(|T_x(y)|)}{g(1)} \frac{T_x(y) \cdot (-z)}{|T_x(y)|} \, d\mu(y) \\
& \leq \mu(\overline{\Bn} \setminus \{ \hat{y} \}) 
= \mu(\overline{\Bn}) - \mu(\{ \hat{y} \}) , 
\end{align*}
which means $\mu( \{ \hat{y} \} ) \leq \frac{1}{2} \mu(\overline{\Bn})$. If equality holds in the displayed formula then $T_x(y)=-z$ for $\mu$-almost every $y \in \overline{\Bn} \setminus \{ \hat{y} \}$, which means $\mu$ has half its mass at $\hat{y}$ and the other half at $T_x^{-1}(-z)$. Thus if $V$ has a vanishing point and $g$ is maximal at $r=1$, then hypothesis \eqref{eq:pointmass} must necessarily hold for all $y \in \partial \Bn$, except when $\mu$ consists of equal point masses concentrated at two boundary points. 

Restricting the measure to the boundary sphere in the last theorem yields a particularly clean result of Hersch type: 
\begin{corollary}[Center of mass on the sphere] \label{cor:spheremu} Assume $\mu$ is a Borel measure on the unit sphere $S^{n-1}, n \geq 2$, satisfying $0 < \mu(S^{n-1}) < \infty$. If  
\begin{equation} \label{eq:pointmasssphere}
\mu( \{ y \} ) < \frac{1}{2} \mu(S^{n-1}) , \qquad y \in S^{n-1} ,
\end{equation}
then a unique point $x=x(\mu) \in \Bn$ exists such that 
\begin{equation} \label{eq:centerofmasssphere}
\int_{S^{n-1}} T_x(y) \, d\mu(y) = 0 .
\end{equation}
That is, pushing forward the measure by $T_x$ results in a center of mass at the origin: 
\[
\int_{S^{n-1}} y \, d[(T_x)_*\mu] = 0 .
\]
This point $x(\mu)$ depends continuously on the measure: if $\mu_k \to \mu$ weakly where the $\mu_k$ are Borel measures on $S^{n-1}$ satisfying \eqref{eq:pointmasssphere}  and $0<\mu_k(S^{n-1})<\infty$, then $x(\mu_k) \to x(\mu)$ as $k \to \infty$. 
\end{corollary}
\begin{proof}%[Proof of \autoref{cor:spheremu}]
Apply \autoref{th:herschszego} with $g(r)=r$ and $\mu$ supported on the sphere, and use that $|T_x(y)|=1$ whenever $|y|=1$ and so $v(T_x(y))=T_x(y)$. The corollary follows. 

\emph{Comment.} The ideas behind this proof were discussed in the Introduction, where the energy method for Hersch's center of mass was summarized in terms of the renormalized energy $\mathcal{E}_{\text{sphere}}(x)$. At an energy minimizing point, the criticality condition $\nabla \mathcal{E}_{\text{sphere}}(x)=0$ implies by formulas \eqref{eq:Kdef} and \eqref{eq:spherekernel} later in the paper that the center of mass equation \eqref{eq:centerofmasssphere} holds. 

\end{proof}
The existence and uniqueness parts of \autoref{cor:spheremu} are due to Douady and Earle \cite[{\S\S}2,11]{DE86} and Millson and Zombro \cite[Lemma 4.11]{MZ96}, as discussed in the Introduction. The result was found again by Girouard, Nadirashvili and Polterovich \cite[Proposition 4.1.5]{GNP09}, for measures without point masses. The latter authors rely on Hersch's topological method for existence, and for uniqueness use that for each point $y$ on the sphere, the component of $T_x(y)$ in direction $x$ exceeds the component of $y$ in direction $x$ (except for the points $y= \pm x/|x|$ that are fixed by $T_x$). They also observe by a short argument with vector fields that uniqueness implies continuous dependence. The continuous dependence proofs in the current paper proceed through properties of the energy, rather than of its gradient field. 

More recently, Biliotti and Ghigi \cite[{\S}7.14 and Theorem 7.6]{BG17} obtained the existence statement of \autoref{cor:spheremu} for the $2$-sphere ($n=3$), as a corollary of their center of mass results for K\"{a}hler manifolds. Energy methods underlie their approach. The kernel is not explicitly formulated. Their Sections 7.17--7.21 are useful in connecting the setting to Hersch's eigenvalue problem, and Problem 1.2 and Theorem 1.3 explain their general goals.  

\smallskip
The original result of Hersch \cite[p.\,1645]{H70} relates to integrals of functions rather than measures: 
\begin{corollary}[Hersch orthogonality on the sphere] \label{cor:hersch}
If $f$ is nonnegative and integrable on $S^{n-1}, n \geq 2$, with $\int_{S^{n-1}} f \, dS > 0$, then a unique point $x \in \Bn$ exists such that each component of $T_x$ is orthogonal to $f$, meaning
\[
\int_{S^{n-1}} T_x(y) f(y) \, dS(y) = 0 .
\]
\end{corollary}
\begin{proof}
Apply \autoref{cor:spheremu} with $d\mu = f \, dS|_{S^{n-1}}$, noting that each point on the sphere $S^{n-1}$ has $\mu$-measure $0$, since $n-1 \geq 1$.  
\end{proof}

The next corollary lives in the plane, where as we mentioned earlier the M\"{o}bius transformation can be written in complex notation as 
\begin{equation} \label{eq:Twzcomplex}
T_w(z) = \frac{z+w}{1+z\overline{w}} , \qquad w \in \D, \quad z \in \overline{\D} .
\end{equation}
Let $\nu(re^{i\theta})=g(r)e^{i\theta}$, which is the complex-valued version of the vector field $v$. Szeg\H{o} \cite{S54} developed the next result with $F$ a conformal map, that is, a biholomorphic map. 
\begin{corollary}[Szeg\H{o} orthogonality on a simply connected domain] \label{cor:szego}
Suppose $f$ is nonnegative and integrable on a simply connected planar domain $\Omega$ with $\int_\Omega f(z) \, |dz|^2 > 0$. If $F : \Omega \to \D$ is continuous and $g(1)>0$ then a point $w \in \D$ exists such that $\nu \circ T_w \circ F$ is orthogonal to $f$, meaning
\[
\int_\Omega \nu(T_w \circ F(z)) f(z) \, |dz|^2 = 0 .
\]
If in addition $g$ is strictly increasing, or else $F$ is a $C^1$-diffeomorphism and $g$ is increasing with $g(r)>0$ for all $0<r<1$, then the point $w$ is unique.
\end{corollary}
\begin{proof}
Apply \autoref{th:herschszego} parts (a) and (b) with $n=2$ and with $d\mu$ being the pushforward of $f(z) \, |dz|^2$ under $F$, that is, with $\mu(A) = \int_{F^{-1}(A)} f(z) \, |dz|^2$ for Borel sets $A \subset \D$, and $\mu=0$ on the unit circle $\partial \D$. For the uniqueness statement, when applying part (b)(ii) notice that if $F$ is a $C^1$-diffeomorphism then the inverse image under $F$ of a hyperbolic geodesic in the disk has Lebesgue measure zero in $\Omega$, and so the geodesic has $\mu$-measure zero; hence $\mu$ is not supported in the geodesic. 
\end{proof}
\begin{corollary}[Weinstock orthogonality on a Jordan curve] \label{cor:weinstock}
Suppose $f$ is nonnegative and integrable on a rectifiable Jordan curve $J$ with $\int_J f \, |dz| > 0$. If $F : J \to \partial \D$ is a homeomorphism then a unique point $w \in \D$ exists such that $T_w \circ F$ is orthogonal to $f$:
\[
\int_J T_w(F(z)) f(z) \, |dz| = 0 .
\]
\end{corollary}
\begin{proof}
Apply \autoref{cor:spheremu} with $n=2$ and $d\mu$ being the pushforward of $f \, |dz|$ under $F$, that is, with $\mu(A) = \int_{F^{-1}(A)} f \, |dz|$ for Borel sets $A \subset \partial \D$. Note each point on the circle has $\mu$-measure $0$, since its inverse image under $F$ is a single point on the Jordan curve. 
\end{proof}

\section{\bf Existence results for signed measures on the closed ball}  \label{sec:resultssignedmeasure}

The existence claims in \autoref{th:herschszego} and \autoref{cor:spheremu}--\autoref{cor:weinstock} hold even when $\mu$ is a \textbf{signed} measure, as the next theorem shows. Assume $g(r)$ is continuous and real valued for $0 \leq r \leq 1$, with $g(0)=0$. 
\begin{theorem}[Existence of center of mass for a signed measure] \label{th:herschszegosigned}\  

(a) Suppose $\mu$ is a signed Borel measure on $\overline{\Bn}$ satisfying $0 < \mu(\overline{\Bn}) \leq |\mu|(\overline{\Bn}) < \infty$. If $g(1)>0$ and $\mu( \{ y \} ) < \frac{1}{2} \mu(\overline{\Bn})$ for all $y \in \partial \Bn$, then $V(x)=0$ for some $x \in \Bn$. 

(b) Suppose $\mu$ is a signed Borel measure on $S^{n-1}, n \geq 2$, satisfying $0 < \mu(S^{n-1}) \leq |\mu|(S^{n-1}) < \infty$. If $\mu( \{ y \} ) < \frac{1}{2} \mu(S^{n-1})$ for all $y \in S^{n-1}$, then $\int_{S^{n-1}} T_x(y) \, d\mu(y) = 0$ for some $x \in \Bn$. Equivalently, pushing forward the signed measure by $T_x$ yields a measure whose center of mass lies at the origin: $\int_{S^{n-1}} y \, d[(T_x)_*\mu] = 0$. 

(c) If $f : S^{n-1} \to \R$ is integrable and $\int_{S^{n-1}} f \, dS \neq 0, n \geq 2$, then a point $x \in \Bn$ exists such that each component of $T_x$ is orthogonal to $f$, meaning $\int_{S^{n-1}} T_x(y) f(y) \, dS(y) = 0$. 

(d) If $f$ is real-valued and integrable on a simply connected planar domain $\Omega$ with $\int_\Omega f(z) \, |dz|^2 \neq 0$, and $F : \Omega \to \D$ is continuous and $g(1)>0$, then a point $w \in \D$ exists such that $\nu \circ T_w \circ F$ is orthogonal to $f$, meaning $\int_\Omega \nu(T_w \circ F(z)) f(z) \, |dz|^2 = 0$. Here $\nu(re^{i\theta})=g(r)e^{i\theta}$. 

(e) If $f$ is real-valued and integrable on a rectifiable Jordan curve $J$ with $\int_J f \, |dz| \neq 0$, and $F : J \to \partial \D$ is a homeomorphism, then a point $w \in \D$ exists such that $T_w \circ F$ is orthogonal to $f$, meaning $\int_J T_w(F(z)) f(z) \, |dz| = 0$. 
\end{theorem}
See \autoref{sec:herschszegoproofsigned} for the proof. In part (a), the hypothesis that the measure of the point $y$ on the sphere is less than half the total measure of the ball is satisfied automatically at all points where $\mu(\{ y \}) \leq 0$, since $\mu(\overline{\Bn})$ is assumed to be positive. Parts (c),(d) and (e) of the theorem provide signed versions of the existence claims in the Hersch, Szeg\H{o}, and Weinstock corollaries, respectively. 

\begin{example*}[Failure of uniqueness for signed measures on the ball]
\autoref{th:herschszegosigned} makes no claims about uniqueness. In fact, uniqueness can fail in part (a) of the theorem, as we show by example in the open unit interval. (The same phenomenon occurs in all dimensions.) Put $a=\tanh 1$ and $b=\tanh 2$, and let 
\[
\mu=-\delta_{-a}+3\delta_0-\delta_a ,
\] 
so that $\mu$ consists of negative point masses at $\pm a$ and a triple point mass at the origin. Observe $\mu(\B)=1>0$, and $\mu$ has no point masses at the boundary points $\pm 1$. Define a continuous, increasing function  
\[
g(r) = 
\begin{cases}
s(r) , & 0 \leq r \leq a, \\
2s(r)-1 & a \leq r \leq b , \\
3 & b \leq r \leq 1 .
\end{cases}
\]
One has $s(a)=1,s(b)=2$, and so $g(a)=1,g(b)=3$. Hence $v(-a)=-1,v(0)=0,v(a)=1,v(b)=3$. Also $T_a(-a)=0, T_a(0)=a, T_a(a)=\tanh(1+1)=b$. Hence the vector field evaluates at $x=0$ and $x=a$ to 
\[
V(0)=(-1)(-1)+0 \cdot 3+1(-1)=0 , \qquad V(a)=0(-1)+1 \cdot 3 +3(-1)=0 .
\]
Thus $V$ vanishes at more than one point. The same reasoning shows that $V(x)=0$ for all $x \in [-a,a]$, and so uniqueness fails badly. 
\end{example*}

\begin{example*}[Failure of uniqueness for signed measures on the sphere]
Uniqueness can fail in \autoref{th:herschszegosigned}(b), that is, for measures on the sphere, by the following example on the unit circle ($n=2$). Using complex notation for $T_w(z)$ as in \eqref{eq:Twzcomplex}, the task is to find a signed Borel measure $\mu$ on $\partial \D$ satisfying $0 < \mu(\partial \D) \leq |\mu|(\partial \D) < \infty$ and $\mu( \{ z \} ) < \frac{1}{2} \mu(\partial \D)$ for all $z \in \partial \D$, such that $\int_{\partial \D} T_w(z) \, d\mu(z) = 0$ for more than one point $w \in \D$. Let 
\[
\mu = \delta_{e^{2\pi i/3}} + \delta_{e^{-2\pi i/3}} - \delta_{e^{\pi i/3}} - \delta_{e^{-\pi i/3}} + \delta_1 
\]
be a sum of positive point masses at $1,e^{2\pi i/3},e^{-2\pi i/3}$, and negative point masses at $e^{\pi i/3},e^{-\pi i/3}$. Clearly $\mu(\partial \D)=1$ and $|\mu|(\partial \D)=5$. The condition $\mu( \{ z \} ) < \frac{1}{2} \mu(\partial \D)$ fails for each point $z$ hosting a positive point mass, but we will fix that issue later by smearing the point masses into continuous densities. 

The vector field is 
\begin{align*}
V(w) 
& = \int_{\partial \D} T_w(z) \, d\mu(z) \\
& = T_w(e^{2\pi i/3}) + T_w(e^{-2\pi i/3}) - T_w(e^{\pi i/3}) - T_w(e^{-\pi i/3}) + T_w(1) 
\end{align*}
for $w \in \D$. Restricting attention to real values $-1<w<1$, we have $T_w(z) = \frac{z+w}{1+zw}$ by \eqref{eq:Twzcomplex}, so that $T_w(\pm 1)=\pm 1$ and $T_w(\overline{z})=\overline{T_w(z)}$. Hence $V(w)$ is real-valued, when $-1<w<1$. 

We will evaluate $V$ at $w=0$, and as $w \to \pm 1$. First 
\[
V(0) = e^{2\pi i/3}+e^{-2\pi i/3} -e^{\pi i/3}-e^{-\pi i/3} + 1 = -1 .
\]
As $w \to 1$ one has $T_w(z) \to 1$ for all $z \in \partial \D \setminus \{ -1 \}$, and so $V(w) \to 1+1-1-1+1=1$. And as $w \to -1$ one has $T_w(z) \to -1$ for all $z \in \partial \D \setminus \{ 1 \}$, and $T_w(1)=1$, so that $V(w) \to -1-1+1+1+1=1$. Thus $V(w)$ changes sign from positive to negative to positive, as $w$ increases from $-1$ to $1$, and so $V$ has at least two zeros. 

Finally, if the point masses are smeared into densities that are symmetric about the real axis and are sufficiently concentrated, then the vector field $V$ maintains the sign-changing property and hence has at least two zeros. This modified measure $\mu$ has no point masses, and so the condition $\mu( \{ z \} ) < \frac{1}{2} \mu(\partial \D)$ is satisfied  for all $z$. 
\end{example*}

\begin{example*}[Failure of existence for signed measures with $\mu(\overline{\Bn})=0$]
The net measure $\mu(\overline{\Bn})$ is assumed nonzero in \autoref{th:herschszegosigned}(a). When it equals zero, the vector field $V$ might not have any vanishing points. For example, in $1$ dimension, if $\mu=\delta_1-\delta_{-1}$ and $g(r)=r$ then $\mu(\overline{\B})=0$ and 
\[
V(x) = T_x(1) - T_x(-1) = \frac{x+1}{1+x} - \frac{x+(-1)}{1+x(-1)} = 2 
\]
for all $x \in (-1,1)$, and so $V$ does not vanish anywhere. This continues to hold if the point masses are smeared into symmetrical, concentrated densities near $\pm 1$.
\end{example*}

\section{\bf Proof of \autoref{th:hyperbolic} --- center of mass on $\Bn$}  \label{sec:hyperbolicproof}	

\subsection*{Part (a) --- Existence} 
Let 
\[
G(r)=\int_0^r g(t) (1-t^2)^{-1} \, dt , \qquad 0 \leq r < 1 ,
\]
so that $G(0)=0$ and $G(r) \to \infty$ as $r \to 1$, by assumption. Put
\[
\Gamma(x) = G(|x|) ,  \qquad x \in \Bn .
\]
This function has gradient
\[
\nabla \Gamma(x) = g(|x|) (1-|x|^2)^{-1} \frac{x}{|x|} =  \frac{1}{1-|x|^2} v(x) .
\]
%(Thus the gradient of $\Gamma$ with respect to the hyperbolic metric equals $(1-|x|^2) v(x)$.) 
More generally, for each fixed $y \in \Bn$ one computes
\begin{equation} \label{eq:gradientgamma}
\nabla_{\! x} \big(\Gamma \circ T_x(y) \big) 
= g(|T_x(y)|) \frac{1}{1-|T_x(y)|^2} \nabla_{\! x} (|T_x(y)|) .
\end{equation}
To evaluate the gradient of $|T_x(y)|$, start with the formula 
\begin{equation} \label{eq:TxyTyx}
|T_x(y)|^2= \frac{|x+y|^2}{1+ 2x \cdot y + |x|^2|y|^2} = |T_y(x)|^2 , \qquad x,y \in \Bn ,
\end{equation}
which holds by direct computation from definition \eqref{eq:Mobius}, Take the gradient to obtain 
\[
\nabla_{\! x} (|T_x(y)|^2) = 2 \frac{1-|T_x(y)|^2}{1-|x|^2} T_x(y) .
\]
Dividing by $2|T_x(y)|$ yields $\nabla_{\! x} (|T_x(y)|)$, and then substituting into \eqref{eq:gradientgamma} gives that  
\begin{equation} \label{eq:Gammagrad}
\nabla_{\! x} \big(\Gamma \circ T_x(y) \big)  
= \frac{1}{1-|x|^2} g(|T_x(y)|) \frac{T_x(y)}{|T_x(y)|} = \frac{1}{1-|x|^2} v(T_x(y)) .
\end{equation}
%
%(Equivalently, the gradient of $x \mapsto \Gamma \circ T_x(y)$ with respect to the hyperbolic metric equals $(1-|x|^2) (v \circ T_x)(y)$.) 

Now define an energy functional
\begin{align*}
E(x) 
& = \int_\Bn (\Gamma \circ T_x) \, d\mu \\
& = \int_\Bn G(|T_x(y)|) \, d\mu(y) , \qquad x \in \Bn .
\end{align*}
This energy $E(x)$ is finite valued and continuously differentiable, since $\mu$ has compact support with finite total measure and the kernel $\Gamma \circ T_x$ is continuously differentiable with respect to $x$, by above. 

Clearly  $E(x) \to \infty$ as $|x| \to 1$, because $|T_x(y)| \to 1$ and $G(r) \to \infty$ as $r \to 1$, and also $\mu(\Bn)>0$. Hence $E(x)$ achieves a minimum at some point $x_c \in \Bn$. 

Differentiating the energy gives that
\begin{align}
\nabla E(x) 
& =  \int_\Bn \nabla_{\! x} \big(\Gamma \circ T_x(y) \big) \, d\mu(y) \notag \\
& = \frac{1}{1-|x|^2} \int_\Bn v(T_x(y)) \, d\mu(y) 
= \frac{1}{1-|x|^2} V(x) \label{eq:hyperbolicgradient}
\end{align}
by above. 
%In other words, the energy has hyperbolic gradient
%%
%\begin{equation} 
%\operatorname{grad}_{hyp} E(x) = (1-|x|^2)V(x) .
%\end{equation}
%%
Thus critical points of the energy are zeros of $V$. In particular, $V$ vanishes at the energy minimizing point $x_c$. 

\smallskip
\subsection*{Part (b) --- Uniqueness by convexity: the geometric method} 
Conditions (i) and (ii) each imply that $g(r)$ is positive and bounded away from zero, as $r \to 1$, so that $\int_0^1 g(r) (1-r^2)^{-1} \, dr = \infty$. Thus part (a) guarantees existence of a point $x_c$ at which $V$ vanishes. 

We will prove below that:
\begin{align} 
\text{if condition (i) holds then the kernel $\Gamma$ is strictly hyperbolically convex;} \label{eq:convexitycondi} \\
\text{if condition (ii) holds then $\Gamma$ is hyperbolically convex along each geodesic,} \notag \\
\text{and the convexity is strict if the geodesic does not pass through the origin.} \label{eq:convexitycondii}
\end{align}
For now, assume these conditions hold. Notice also that $|T_x(y)|=|T_y(x)|$ by \eqref{eq:TxyTyx}, and so we may interchange $x$ and $y$ in the energy integral to get $E(x) = \int_\Bn \Gamma\big(T_y(x)\big) \, d\mu(y)$. Recall the M\"{o}bius transformation $T_y(\cdot)$ is a hyperbolic isometry, and thus preserves convexity along geodesics. 

If condition (i) holds then $x \mapsto \Gamma\big(T_y(x)\big)$ is strictly hyperbolically convex by \eqref{eq:convexitycondi}, for each $y \in \Bn$. Integrating with respect to $d\mu(y)$ gives that $E(x)$ is strictly hyperbolically convex, and hence its critical point $x_c$ is unique. 

If condition (ii) holds, then the same argument gives hyperbolic convexity of $E(x)$ along each geodesic $\gamma$, and the convexity is strict unless the geodesic $T_y(\gamma)$ passes through $0$ for $\mu$-almost every $y \in \Bn$. That exceptional case would imply $\gamma$ contains the point $T_y^{-1}(0)=-y$ for $\mu$-almost every $y$, and so $\mu$ would be supported in the geodesic $-\gamma$. Such a circumstance is forbidden in condition (ii), and so $E(x)$ is strictly hyperbolically convex along each geodesic, implying uniqueness of the critical point $x_c$.

\smallskip
It remains to prove implications \eqref{eq:convexitycondi} and \eqref{eq:convexitycondii}. First we show $\Gamma$ is hyperbolically convex if $g$ is increasing and $g(r)>0$ for all $r>0$, which holds true under either condition (i) or condition (ii). Recall that $s = s(r) = \frac{1}{2} \log \frac{1+r}{1-r}$ measures hyperbolic distance from the origin to a point at euclidean radius $r$, with $ds/dr = (1-r^2)^{-1}$. Define a new function $\widetilde{G}(s)$ for $0 \leq s < \infty$ by
\[
G(r) = \widetilde{G}(s(r)) , \qquad 0 \leq r < 1 .
\]
Then 
\[
\Gamma(x) = G(|x|) = \widetilde{G} \big( d_\Bn(x,0) \big) . 
\]
And since $\widetilde{G}^\prime(s) = (1-r^2)G^\prime(r)=g(r)>0$ and $g$ is increasing, we see $\widetilde{G}(s)$ is convex and strictly increasing. 

To show $\Gamma$ is hyperbolically convex, consider points $x_0,x_1 \in \Bn$ with $x_0 \neq x_1$, and write $\gamma$ for the geodesic joining the two points. Let $0<\e<1$ and write $x_\e$ for the point along $\gamma$ whose distance from $x_0$ is $\e d_\Bn(x_0,x_1)$ and distance from $x_1$ is $(1-\e)d_\Bn(x_0,x_1)$. Observe 
\begin{align}
\Gamma(x_\e) 
& = \widetilde{G}\big( d_\Bn(x_\e,0)\big) \notag \\
& \leq \widetilde{G}\big( (1-\e) d_\Bn(x_0,0) + \e d_\Bn(x_1,0) \big) \qquad \text{since $\widetilde{G}$ is increasing} \label{eq:hypconvexGamma} \\
& \hspace*{1cm} \text{and $d_\Bn(\cdot,0)$ is hyperbolically convex by \autoref{le:hypgeo},} \notag \\
& \leq (1-\e)\widetilde{G}\big(d_\Bn(x_0,0)\big) + \e \widetilde{G}\big(d_\Bn(x_1,0)\big) \qquad \text{by convexity of $\widetilde{G}$} \label{eq:hypconvexGamma2} \\
& = (1-\e)\Gamma(x_0) + \e \Gamma(x_1) . \notag
\end{align}
Hence $\Gamma$ is hyperbolically convex. 

We must strengthen the conclusion to strict hyperbolic convexity. Note first that since $\widetilde{G}$ is strictly increasing, equality in \eqref{eq:hypconvexGamma} would imply that the hyperbolic convexity of $d_\Bn(\cdot,0)$ is nonstrict along $\gamma$, which by \autoref{le:hypgeo} would imply that $x_0$ and $x_1$ point in the same direction (that is, lie on the same ray from the origin).

If condition (i) holds then $\widetilde{G}^\prime(s)=g(r)$ is strictly increasing and hence $\widetilde{G}$ is strictly convex. If equality holds in \eqref{eq:hypconvexGamma2} then this strict convexity implies $|x_0|=|x_1|$. Since $x_0 \neq x_1$ by assumption, the vectors must point in different directions, and so inequality \eqref{eq:hypconvexGamma} is strict. Hence $\Gamma$ is strictly hyperbolically convex, proving implication \eqref{eq:convexitycondi}. 

\smallskip
Suppose condition (ii) holds. To prove implication \eqref{eq:convexitycondii} we must show that if the hyperbolic convexity of $\Gamma$ along some geodesic is not strict, then that geodesic passes through the origin. For this, simply observe that if equality holds in \eqref{eq:hypconvexGamma} for some $x_0 \neq x_1$ then (by \autoref{le:hypgeo}) the points $x_0$ and $x_1$ must lie on some ray from the origin, and so the geodesic that passes through the points must also pass through the origin. 

\smallskip
\subsection*{Part (b) --- Uniqueness by convexity: the analytic method} 
As in the geometric method above, the task reduces to proving the hyperbolic convexity implications \eqref{eq:convexitycondi} and \eqref{eq:convexitycondii} for the kernel $\Gamma$. This time we prove them by analytic techniques. 

Parameterize a hyperbolic geodesic $\gamma$ by $x(t)$, where $t$ is euclidean arclength along the curve. The derivative of $\Gamma$ along the geodesic with respect to hyperbolic arclength is 
\[
(1-|x(t)|^2) \frac{d\ }{dt} \Gamma \big( x(t) \big) = g\big( |x(t)| \big) \frac{d\ }{dt} |x(t)| = v\big(x(t)\big) \cdot x^\prime(t) ,
\]
where we used that $G^\prime(r)=(1-r^2)^{-1} g(r)$. The desired implications \eqref{eq:convexitycondi} and \eqref{eq:convexitycondii} can therefore be rephrased as follows:
\begin{align} 
\text{if condition (i) holds then $v\big(x(t)\big) \cdot x^\prime(t)$ is strictly increasing;} \label{eq:convexitycondiii} \\
\text{if condition (ii) holds then $v\big(x(t)\big) \cdot x^\prime(t)$ is increasing, and is} \notag \\
\text{strictly increasing if the geodesic does not pass through the origin.} \label{eq:convexitycondiv}
\end{align}

First suppose $g$ is increasing and $g(r)>0$ for all $r>0$, which holds under both conditions (i) and (ii). Suppose further that $\gamma$ does not pass through the origin. In particular, this means the dimension $n$ is greater than $1$. We will show $v\big(x(t)\big) \cdot x^\prime(t)$ is strictly increasing. After a suitable rotation to place the geodesic into the $x_1 x_2$-plane symmetrically about the $x_1$-axis, the geodesic can be taken as the arc within the unit disk of the circle having radius $b$ centered at $(x_1,x_2)=(a,0)$, where $a^2=b^2+1$. A parameterization of this geodesic in terms of euclidean arclength $t$ is 
\[
x(t) = \big( a- b \cos (t/b), b \sin (t/b) \big) , \qquad |t| < b \arccos \frac{b}{a} ,
\]
so that 
\[
|x(t)| = \sqrt{a^2 - 2ab \cos (t/b) + b^2} ,
\]
which equals $1$ when $|t| = b \arccos (b/a)$. The second derivative is
\[
\frac{d^2\ }{dt^2} \, |x(t)| = a^2 \frac{(a/b-\cos (t/b))(\cos (t/b) - b/a)}{(a^2 - 2ab \cos t/b + b^2)^{3/2}} ,
\]
which is positive since $a/b>1$ and because $\cos (t/b) > b/a$ when $|t| < b \arccos (b/a)$. Thus $|x(t)|$ is a strictly convex function of $t$ when $|t| < b \arccos (b/a)$, and has its minimum at $t=0$. In particular, $|x(t)|$ is positive and decreasing when $t \in (-b \arccos(b/a),0)$ and is positive and increasing when $t \in (0,b \arccos(b/a))$. Hence $g\big( |x(t)| \big)$ has the same properties, because $g(r)$ is positive and increasing for $r>0$. Further, the first derivative $(d/dt) |x(t)|$ is negative and strictly increasing when $t \in (-b \arccos (b/a),0)$, and positive and strictly increasing when $t \in (0,b \arccos (b/a))$, by using the strict convexity. Putting these facts together shows that $g\big( |x(t)| \big) \frac{d\ }{dt} |x(t)|=v\big(x(t)\big) \cdot x^\prime(t)$ is a strictly increasing function of $t$. 

To finish proving \eqref{eq:convexitycondiii} and \eqref{eq:convexitycondiv}, we show that if $\gamma$ does pass through the origin, then $g\big( |x(t)| \big) \frac{d\ }{dt} |x(t)|$ is increasing provided $g$ is increasing (condition (ii)) and is strictly increasing if $g$ is strictly increasing (condition (i)). Indeed, after a suitable rotation to the $x_1x_2$-plane, the geodesic can be parameterized as the line segment $x(t)=(0,t)$ for $t \in (-1,1)$, so that $g\big( |x(t)| \big) \frac{d\ }{dt} |x(t)| = \sign(t) g(|t|)$, which is an increasing function of $t$ and is strictly increasing if $g$ is strictly increasing. 

\smallskip
\subsection*{Part (c) --- Continuous dependence} The following proof is a slight adaptation of the euclidean case in \cite[Theorem 1(c)]{L20e}. 

Part (c) assumes that either (i) holds or else (ii) holds for $\mu, \mu_1,\mu_2,\mu_3,\dots$. Hence by parts (a) and (b), we may write $x_c(\mu)$ for the unique minimum point of the energy $E$ corresponding to the measure $\mu$, and $x_c(\mu_k)$ for the unique minimum point of the energy $E_k$ corresponding to the measure $\mu_k$. 

The measures $\mu_k$ and $\mu$ are assumed to be supported in some fixed compact set $Y \subset \Bn$, and the weak convergence $\mu_k \to \mu$ implies that $\mu_k(Y) \to \mu(Y)$. Let $X$ be an arbitrary compact set in $\Bn$. The kernel $\Gamma(T_x(y))$ is uniformly continuous and bounded for $(x,y) \in X \times Y$, and it follows easily that the family $\{ E_k(x) \}_{k=1}^\infty$ is uniformly equicontinuous on $X$. Therefore the weak convergence $\mu_k \to \mu$ as $k \to \infty$ implies that $E_k(x) \to E(x)$ pointwise and also, after a short argument using equicontinuity, uniformly for $x \in X$. 
%
%Proof of uniform convergence:
%An ``$\e/3$'' argument together with the pointwise convergence $E_k(x) \to E(x)$ implies that $E_k \to E$ uniformly on $X$.
%

Let $\e>0$, and denote by $B$ the open ball of radius $\e$ centered at $x_c(\mu)$, with $\e$ chosen small enough that $\overline{B} \subset \Bn$. The strict energy minimizing property of $x_c(\mu)$ implies 
\[
E(x_c(\mu)) < \min_{x \in \partial B} E(x) ,
\]
and so (by choosing $X=\partial B$) we deduce from uniform convergence that 
\[
E_k(x_c(\mu)) < \min_{x \in \partial B} E_k(x) 
\]
for all large $k$. Consequently, the open ball $B$ contains a local minimum point for the energy $E_k$, when $k$ is large. This local minimum must be the global minimum point $x_c(\mu_k)$, by strict hyperbolic convexity of the energy. Since $\e$ was arbitrary, we conclude $x_c(\mu_k) \to x_c(\mu)$ as $k \to \infty$, giving continuous dependence.

\section{\bf Proof of \autoref{cor:weinfoldhyp} --- Orthogonality with a hyperbolic fold}  \label{sec:weinfoldhypproof}

Existence and uniqueness follow from \autoref{th:hyperbolic} with $\mu$ being the pushforward under $F$ of the measure $f(y) (1-|y|^2)^{-n} dy$ on $\Omega$. This $\mu$ is not supported in any hyperbolic geodesic, and so condition (ii) holds as needed in part (b) of the theorem. 

For continuous dependence, in order to check the hypotheses of \autoref{th:hyperbolic}(c) we show $\mu$ is (weakly sequentially) continuous with respect to the parameters $(p,t)$ of the hyperbolic halfspace $H(p,t)$. So suppose $p_k \to p$ in $S^{n-1}$ and $t_k \to t$ in $(-1,1)$. Let $F_k$ be the fold map coresponding to $H(p_k,t_k)$, and $\mu_k$ be the pushforward under $F_k$ of the measure $f(y) (1-|y|^2)^{-n} dy|_\Omega$. The image of $\Omega$ under $F_k$ is compactly contained in $\Bn$ independently of $k$, which means the measures $\mu_k$ are all supported in some fixed compact set in the ball. To apply part (c) of the theorem, we have only to show $\mu_k \to \mu$ weakly. So take a continuous function $\psi(y)$ on $\Bn$, and note
\begin{align*}
\int_\Bn \psi \, d\mu_k
& = \int_\Omega \psi(F_k(y)) f(y) (1-|y|^2)^{-n} dy \\
& \to \int_\Omega \psi(F(y)) f(y) (1-|y|^2)^{-n} dy = \int_\Bn \psi \, d\mu 
\end{align*}
by using locally uniform convergence of $F_k$ to $F$, or else by dominated convergence.

\section{\bf Renormalizing the energy for the closed ball}  \label{sec:renormalizedenergy}	

The measure $\mu$ in \autoref{th:herschszego} can live on the whole closed ball. If it has positive mass on the boundary then the energy $E(x)$ used in proving \autoref{th:hyperbolic} will equal $+\infty$ at every $x$, because $G(|T_x(y)|)=G(1)=\infty$ whenever $y \in \partial \Bn$. To rescue the energy method from this fate, we investigate in this section the properties of a renormalized kernel and energy. Then the next section proves \autoref{th:herschszego}. 

Consider a continuous function $g(r)$ for $0 \leq r \leq 1$, with $g(0)=0$ and $g(1)>0$. Without loss of generality, we may assume $g(1)=1$. Recall the weighted antiderivative $G(r)=\int_0^r g(t) (1-t^2)^{-1} \, dt$ for $0 \leq r < 1$. Note $G(r) \to \infty$ as $r \to 1$, and remember $\Gamma(x) = G(|x|)$ for $x \in \Bn$. Define a \textbf{renormalized kernel} 
\begin{equation} \label{eq:Kdef}
K(x,y) = 
\begin{cases}
\Gamma(T_x(y)) - \Gamma(y) & \text{for\ } |x|<1, \ |y| < 1 , \\
{\displaystyle \frac{1}{2} \log \frac{|x+y|^2}{1-|x|^2} } & \text{for\ } |x|<1, \ |y| = 1 .
\end{cases}
\end{equation}
That is, the renormalization procedure consists of subtracting $\Gamma(y)$ when $|y|<1$, and extending continuously to the spatial boundary when $|y|=1$, as the next result explains.
\begin{lemma} \label{le:renormkernel}
The kernel $K : \Bn \times \overline{\Bn} \to \R$ is continuous, and its gradient with respect to $x$ is
\begin{equation} \label{eq:Kgrad}
\nabla_{\! x} K(x,y) = \frac{1}{1-|x|^2} (v \circ T_x)(y) .
\end{equation}
\end{lemma}
\begin{proof}
Clearly $K$ is continuous on the subset where $|x|<1,|y|<1$, and also where $|x|<1,|y|=1$. The issue is to prove continuity when $|x|<1$ and $|y| \to 1$. 

So suppose $|x|<1,|y|<1$ and $(x,y) \to (x_0,y_0)$, where $|x_0|<1,|y_0|=1$. First we show that
\begin{equation} \label{eq:distrelation}
d_\Bn \! \big( 0, T_x(y) \big) - d_\Bn \! \big( 0, y \big) \to \frac{1}{2} \log \frac{|x_0+y_0|^2}{1-|x_0|^2} .
\end{equation}
Indeed, 
\begin{align*}
& d_\Bn \! \big( 0, T_x(y) \big) - d_\Bn \! \big( 0, y \big) \\
& = \frac{1}{2} \log \frac{1+|T_x(y)|}{1-|T_x(y)|} - \frac{1}{2} \log \frac{1+|y|}{1-|y|} \\ 
& = \frac{1}{2} \log \frac{1-|y|^2}{1-|T_x(y)|^2} + \log \frac{1+|T_x(y)|}{1+|y|} \\ 
& = \frac{1}{2} \log \frac{|x|^2 |y|^2 + 2 x \cdot y + 1}{1-|x|^2} + \log \frac{1+|T_x(y)|}{1+|y|} \qquad \text{by using \eqref{eq:TxyTyx} for $|T_x(y)|^2$} \\ 
& \to \frac{1}{2} \log \frac{|x_0+y_0|^2}{1-|x_0|^2} 
\end{align*}
as $(x,y) \to (x_0,y_0)$, by recalling that $|y_0|=1$ and hence $|T_x(y)| \to |T_{x_0}(y_0)| = 1$.

Next we establish the limit:
\begin{equation} \label{eq:Kconvergence}
K(x,y) = \Gamma(T_x(y)) - \Gamma(y) \to \frac{1}{2} \log \frac{|x_0+y_0|^2}{1-|x_0|^2} ,
\end{equation}
which is the desired continuity claim. The proof will use the elementary fact that 
\begin{equation} \label{eq:meanlimit}
\fint_a^b g(r) \, ds(r) \to g(1) = 1
\end{equation}
as $a \to 1$, where $0<a<b<1$ and $ds(r)= (1-r^2)^{-1} \, dr$ is the hyperbolic arclength element.

Suppose to begin with that the logarithm on the right side of \eqref{eq:distrelation} is positive, so that we may suppose $|T_x(y)|>|y|$ as we pass to the limit. By starting with the definition of $\Gamma$ and then multiplying and dividing to get a mean value integral, we find
\begin{align*}
\Gamma(T_x(y)) - \Gamma(y) 
& = \int_{|y|}^{|T_x(y)|} g(r) (1-r^2)^{-1} \, dr \\
& = \int_{|y|}^{|T_x(y)|} \, ds(r) \fint_{|y|}^{|T_x(y)|} g(r) \, ds(r) \quad \ \text{where $e_1=(1,0,\dots,0)$} \\
& \to \frac{1}{2} \log \frac{|x_0+y_0|^2}{1-|x_0|^2}
\end{align*}
as $(x,y) \to (x_0,y_0)$, where in the final line we applied \eqref{eq:distrelation} and \eqref{eq:meanlimit}. The argument is almost identical if the logarithm on the right side of \eqref{eq:Kconvergence} is negative. Lastly, if the logarithm is zero then the proof above continues to apply except at points $(x,y)$ for which $|T_x(y)|=|y|$; but those points already satisfy $\Gamma(T_x(y)) - \Gamma(y) = 0 $, which is the desired limiting value. This completes the proof of \eqref{eq:Kconvergence}, and so the kernel $K$ is continuous. 

When $|y|<1$, the gradient formula \eqref{eq:Kgrad} follows from \eqref{eq:Gammagrad}, since the term $-\Gamma(y)$ in the kernel is independent of $x$.  

Now suppose $|y|=1$. A direct computation yields that 
\begin{equation} \label{eq:spherekernel} 
\nabla_{\! x} K(x,y) 
= \nabla_{\! x} \! \left( \frac{1}{2} \log \frac{|x+y|^2}{1-|x|^2} \right) 
= \frac{1}{1-|x|^2} T_x(y) , 
\end{equation}
where the assumption that $|y|=1$ is used in this calculation to convert $|x+y|^2$ into $1+2x \cdot y + |x|^2|y|^2$, which is the denominator of $T_x(y)$. Since $|T_x(y)|=1$ when $|y|=1$, we have $g(1)T_x(y)=v(T_x(y))$, and so we deduce  \eqref{eq:Kgrad}. 
\end{proof}
A useful transformation property of the kernel is the following. 
\begin{lemma}[M\"{o}bius action on the renormalized kernel when $|y|=1$] \label{le:buseman}
\[
K(T_x(z),y)=K(z,T_x(y))+K(x,y) , \qquad x,z \in \Bn, \quad y \in \partial \Bn .
\]
\end{lemma}
\begin{proof}
The formula is a known property of the Busemann function for the hyperbolic ball. An elegant geometric proof was given by Helgason \cite[p.\,83, (46)]{H81}. To prove the lemma analytically, notice the logarithmic definition of the kernel from \eqref{eq:Kdef} reduces the problem to showing
\[
 \frac{|T_x(z)+y|^2}{1-|T_x(z)|^2} =  \frac{|z+T_x(y)|^2}{1-|z|^2} \, \frac{|x+y|^2}{1-|x|^2} .
\]
Since 
\[
1-|T_x(z)|^2= \frac{(1-|x|^2)(1-|z|^2)}{1 + 2 x \cdot z + |x|^2 |z|^2} 
\]
by \eqref{eq:TxyTyx}, the desired formula simplifies to
\[
|T_x(z)+y|^2 (1 + 2 x \cdot z + |x|^2 |z|^2) =  |z+T_x(y)|^2 |x+y|^2 ,
\]
which can then be verified by substituting the definition of the M\"{o}bius transformations from \eqref{eq:Mobius}, and using that $|y|=1$.
\end{proof}
Next we show the renormalized kernel is strictly hyperbolically convex with respect to $x$, when $|y|=1$, except along certain directions. Non-strict hyperbolic convexity follows by passing to the limit from the case $|y|<1$ that was treated in \autoref{sec:hyperbolicproof}, but the strictness of the convexity will be important in the next section.  
\begin{proposition} \label{pr:renormconvex}
For each $y \in \partial \Bn$, the kernel 
\[
K(x,y) = \frac{1}{2} \log \frac{|x+y|^2}{1-|x|^2}
\]
is strictly hyperbolically convex with respect to $x \in \Bn$ (that is, strictly convex with respect to hyperbolic arclength along each geodesic), except that it is hyperbolically linear along geodesics emanating from the boundary point $-y$.  
\end{proposition}
The proposition is due to Millson and Zombro \cite[Corollary 4.4]{MZ96}. A proof is provided below to keep the presentation self-contained. 
\begin{proof}
Fix $x$ and $y$ with $|x|<1$ and $|y|=1$, and consider a geodesic passing through $x$. It suffices to show the second derivative of the kernel with respect to hyperbolic arclength along the geodesic is positive at $x$, or else, is zero if the geodesic has one endpoint at $-y$. The geodesic can be parameterized as $t \mapsto T_x(tz)$ for some unit vector $z$, since the isometry $T_x$ maps geodesics through the origin ($t \mapsto tz$) to geodesics through $x$. 

By invoking \autoref{le:buseman} to evaluate $K(T_x(tz),y)$, we see the goal is equivalent to showing the second hyperbolic derivative of $t \mapsto K(tz,T_x(y))$ is positive at $t=0$, or else, is zero if the geodesic $t \mapsto tz$ has one endpoint at $(T_x)^{-1}(-y)$. That endpoint case means 
\[
\pm z = (T_x)^{-1}(-y) = T_{-x}(-y) = -T_x(y) ,
\]
by using the definition of $T_x$ in \eqref{eq:Mobius}, and so the endpoint case means $z = \pm T_x(y)$. 

Defining 
\[
h(t) = K(tz,T_x(y)) = \frac{1}{2} \log \frac{|tz+T_x(y)|^2}{1-|tz|^2} , \qquad -1<t<1 ,
\]
for positivity of the second hyperbolic derivative we want  
\[
(1-t^2) \left( (1-t^2) h^\prime(t) \right)^\prime \Big|_{t=0} > 0 ,
\]
or else, $=0$ if $z = \pm T_x(y)$. That is, we want $h^{\prime \prime}(0) > 0$, or else, $=0$ if $z = \pm T_x(y)$. Since $|T_x(y)|=1$, an easy calculation gives that 
\[
h(t) = t \, T_x(y) \cdot z + t^2 \big (1-(T_x(y) \cdot z )^2 \big) + O(t^3) .
\]
Remembering $T_x(y)$ and $z$ are unit vectors, we deduce that $h^{\prime \prime}(0) \geq 0$, with equality if and only if $z = \pm T_x(y)$. This completes the proof. 
\end{proof}

\section{\bf Proof of \autoref{th:herschszego} --- center of mass on the closed ball}  \label{sec:herschszegoproof}	

Without loss of generality, we may assume $g(1)=1$ throughout the proof. 

First, observe that the renormalized energy 
\[
\mathcal{E}(x) = \int_{\overline{\Bn}} K(x,y) \, d\mu(y) , \qquad x \in \Bn ,
\]
is finite valued and depends continuously on $x \in \Bn$, since the renormalized kernel $K$ defined in \eqref{eq:Kdef} is continuous by \autoref{le:renormkernel}, and the measure $\mu$ is finite. 

\smallskip
\emph{Aside.} The definition of $K(x,y)$ implies the renormalized energy is formally equal to $\mathcal{E}(x)=E(x)-E(0)$. This formal approach cannot be used as a definition, though, since both $E(x)$ and $E(0)$ might be infinite. That is why we first renormalized the kernel in \autoref{sec:renormalizedenergy} and showed that it extends continuously to $y \in \partial \Bn$, before defining $\mathcal{E}(x)$ as above.

\smallskip
\textbf{Part (a) [Existence].} Differentiating through the integral with the help of \autoref{le:renormkernel} shows that
\begin{align*}
\nabla \mathcal{E}(x) 
& =  \int_{\overline{\Bn}} \nabla_{\! x} K(x,y) \, d\mu(y) \\
& = \frac{1}{1-|x|^2} \int_{\overline{\Bn}} v(T_x(y)) \, d\mu(y) 
= \frac{1}{1-|x|^2} V(x) .
\end{align*}
%
%That is, the hyperbolic gradient of the renormalized energy is 
%\[
%\operatorname{grad}_{hyp} \mathcal{E}(x) = (1-|x|^2) V(x) .
%\] 

Critical points of the renormalized energy are zeros of $V$. To show the energy has a minimum point, and hence a critical point, we will prove $\mathcal{E}(x) \to \infty$ as $|x| \to 1$.

We use without comment the following facts: $v(r\hat{y})=g(r)\hat{y}$ is continuous on $\overline{\Bn}$;  when $r=1$ we have $g(1)=1$ and so $v(\hat{y})=\hat{y}$; and boundedness of $v$ implies boundedness of $V$. Remember always that $x \in \Bn$. 

First we show the gradient field $V$ points outward when $x$ is near the boundary of the ball. Specifically, we will show a number $\delta>0$ exists such that 
\[
z \cdot V(rz) \geq \delta , \qquad z \in \partial \Bn ,
\]
whenever $r$ is sufficiently close to $1$, say for $\rho < r < 1$. For this it is enough to show $\liminf_{x \to z} z \cdot V(x) > 0$, for each $z \in \partial \Bn$. As $x \to z$, the M\"{o}bius transformation $T_x$ pushes the whole ball except for the antipodal point toward $z$, meaning  $T_x(y) \to z$ for all $y \in \overline{\Bn} \setminus \{ -z \}$, as one may check directly from the definition \eqref{eq:Mobius}. Hence  
\begin{align*}
V(x)
& = \int_{\overline{\Bn} \setminus \{ -z \}} v(T_x(y)) \, d\mu(y) + v(T_x(-z)) \mu(\{ -z \}) \\
& = \left( v(z) \mu(\overline{\Bn} \setminus \{ -z \}) + o(1) \right) + v(T_x(-z)) \mu(\{ -z \}) \quad \text{as $x \to z$} \\
& = z \mu(\overline{\Bn} \setminus \{ -z \}) + T_x(-z) \mu(\{ -z \}) + o(1) ,
\end{align*}
and so  
\begin{align*}
\liminf_{x \to z} z \cdot V(x)
& \geq \mu(\overline{\Bn} \setminus \{ -z \}) - \mu(\{ -z \}) \\
& = 2\e(z) > 0 ,
\end{align*}
where $\e(z) = \frac{1}{2} \mu(\overline{\Bn}) - \mu( \{ -z \} )$, noting that $\e(z)$ is positive by hypothesis \eqref{eq:pointmass}. 

Now that $V$ points outward near the boundary, we could invoke Brouwer's fixed point theorem to conclude that $V$ must vanish somewhere in the ball. Instead, we do a little extra work to show $\mathcal{E}$ approaches infinity at the boundary of the ball. 

By integrating along the ray in direction $z$, one finds when $\rho<r<1$ that 
\begin{align*}
\mathcal{E}(rz)-\mathcal{E}(\rho z)
& = \int_\rho^r z \cdot \nabla \mathcal{E} (tz) \, dt \\
& \geq \delta \int_\rho^r \frac{1}{1-t^2} \, dt \qquad \text{since $\nabla \mathcal{E}(x) = V(x)/(1-|x|^2)$} \\
& = \delta(s(r) - s(\rho)) .
\end{align*}
Hence 
\[
\min_{|x|=r} \mathcal{E} \geq \min_{|x|=\rho} \mathcal{E} - \delta s(\rho) + \delta s(r) \to \infty
\]
as $r \to 1$, because $s(r) \to \infty$. Thus $\mathcal{E}(x)$ tends to infinity as $|x| \to 1$, which completes the proof.  

\smallskip
\textbf{Part (b) [Uniqueness].} Conditions (i) and (ii) each imply that $0 < g(r) \leq g(1)=1$ when $0 < r \leq 1$, and so part (a) guarantees existence of a point $x_c$ at which $\mathcal{E}$ is minimal and hence $V$ vanishes. We will show that critical point is unique by proving $\mathcal{E}$ is strictly hyperbolically convex. 

Start by decomposing the renormalized energy into the contributions from the open ball and the boundary sphere:
\[
\mathcal{E}(x) = \mathcal{E}_{\text{ball}}(x) + \mathcal{E}_{\text{sphere}}(x)
\]
where
\begin{align*}
\mathcal{E}_{\text{ball}}(x) & = \int_\Bn K(x,y) \, d\mu(y) = \int_\Bn \left( \Gamma(T_x(y)) - \Gamma(y) \right) d\mu(y) , \\ 
\mathcal{E}_{\text{sphere}}(x) 
& = \int_{\partial \Bn} K(x,y) \, d\mu(y) = \int_{S^{n-1}} \frac{1}{2} \log \frac{|x+y|^2}{1-|x|^2} \, d\mu(y) .
\end{align*}
This sphere energy was defined already in \eqref{eq:introrenormalized}, in the Introduction. 

\smallskip
Suppose condition (i) holds. When $y \in \Bn$ is fixed, assertion \eqref{eq:convexitycondi} in the proof of \autoref{th:hyperbolic}(b) shows that the map $x \mapsto \Gamma(T_x(y))$ is strictly hyperbolically convex along each geodesic. Hence so is the map $x \mapsto K(x,y) = \Gamma(T_x(y)) - \Gamma(y)$. Therefore $\mathcal{E}_{\text{ball}}(x)$ is strictly hyperbolically convex unless $\mu(\Bn)=0$. Further, when $y \in \partial \Bn$ the map $x \mapsto K(x,y)$ is strictly hyperbolically convex along each geodesic by \autoref{pr:renormconvex}, except that it is hyperbolically linear along geodesics emanating from the boundary point $-y$. Integrating with respect to $y$ shows that $\mathcal{E}_{\text{sphere}}(x)$ is hyperbolically convex, and hence $\mathcal{E}(x)$ is hyperbolically convex. 

Suppose the convexity of $\mathcal{E}$ is non-strict along some hyperbolic geodesic $\gamma$. Then by the preceding paragraph, $\mu(\Bn)=0$ and on $\partial \Bn$ the measure $\mu$ must be supported at the endpoints of $-\gamma$, since for every other point $y \in \partial \Bn$ one has strict hyperbolic convexity of $x \mapsto K(x,y)$ along $\gamma$. Hence $\mu$ must consist of point masses at one or two points on the sphere, which is impossible since hypothesis \eqref{eq:pointmass} says that no point on the boundary can hold half or more of the total mass. Thus $\mathcal{E}(x)$ is strictly hyperbolically convex. 

\smallskip
Now suppose condition (ii) holds. Arguing similarly to condition (i) above, we use the proof of \autoref{th:hyperbolic}(b) to show that $\mathcal{E}_{\text{ball}}(x)$ is hyperbolically convex, and that if the convexity is not strict along some geodesic $\gamma$ then $\mu |_\Bn$ is supported in $-\gamma$. (This exceptional case includes the possibility that $\mu \equiv 0$ on $\Bn$.) Meanwhile, \autoref{pr:renormconvex} guarantees that $\mathcal{E}_{\text{sphere}}(x)$ is hyperbolically convex, and that if the convexity is not strict along $\gamma$ then $\mu$ restricted to the sphere is supported at the endpoints of $-\gamma$. (This exceptional case includes the possibility that $\mu \equiv 0$ on $\partial \Bn$.) 

Adding the two contributions to the energy shows $\mathcal{E}(x)$ is hyperbolically convex. The convexity must be strict along each geodesic $\gamma$ because otherwise the previous paragraph would force $\mu$ to be supported in the closure of $-\gamma$, which is impossible under condition (ii). 

\smallskip
\textbf{Part (c) [Continuous dependence].} Since part (c) assumes either (i) holds or else (ii) holds for $\mu, \mu_1,\mu_2,\mu_3,\dots$, the hypotheses of parts (a) and (b) are satisfied with respect to the measures $\mu$ and $\mu_k$. Write $x_c(\mu)$ for the unique minimum point of the renormalized energy $\mathcal{E}$ corresponding to the measure $\mu$, and $x_c(\mu_k)$ for the unique minimum point of the renormalized energy $\mathcal{E}_k(x) = \int_{\overline{\Bn}} K(x,y) \, d\mu_k(y)$ corresponding to the measure $\mu_k$. 

The weak convergence $\mu_k \to \mu$ implies that $\mu_k(\overline{\Bn}) \to \mu(\overline{\Bn})$. Let $X$ be an arbitrary compact set in $\Bn$. The kernel $K(x,y)$ is uniformly continuous and bounded for $(x,y) \in X \times \overline{\Bn}$, and it follows easily that the family $\{ E_k(x) \}_{k=1}^\infty$ is uniformly equicontinuous on $X$. Therefore the weak convergence $\mu_k \to \mu$ implies that $\mathcal{E}_k(x) \to \mathcal{E}(x)$ pointwise and then, after a short argument using equicontinuity, uniformly for $x \in X$. 
%
%Proof of uniform convergence:
%An ``$\e/3$'' argument together with the pointwise convergence $\mathcal{E}_k(x) \to \mathcal{E}(x)$ implies that $\mathcal{E}_k \to \mathcal{E}$ uniformly on $X$.
%
Now complete the proof as for \autoref{th:hyperbolic}(c), except changing $E$ to $\mathcal{E}$. \qed

\medskip
\noindent 
\emph{Comment on the literature in the spherical case.} The existence and uniqueness proofs above rely on the facts that $\mathcal{E}(x) \to \infty$ as $|x| \to 1$ and $\mathcal{E}(x)$ is strictly hyperbolically convex. For measures supported wholly on the boundary sphere, these facts are due to Millson and Zombro \cite[Propositions 4.5 and 4.8]{MZ96}.  

The proof in their paper that $\mathcal{E}_{\text{sphere}}(x) \to \infty$ relies on \cite[Lemma 4.10]{MZ96}. That lemma seems incorrect as written, because the energy cannot be increasing along all rays out of the origin, unless the minimum point is exactly at the origin. The error in the lemma's proof seems to reside in the assertion that ``$B(\nu)$ clearly lies in the half-ball defined by $u \cdot \gamma < 0$.'' For this assertion the authors seem to want most of the mass of $\nu$ to lie in the half-ball, but there is no reason for that to be the case. Fortunately, the lemma can be rescued by considering only points $x \in \gamma$ that lie sufficiently close to $u$.

\section{\bf Proof of \autoref{th:herschszegosigned} --- center of mass for signed measures}  \label{sec:herschszegoproofsigned}	

\smallskip Part (a). The proof goes exactly as for the proof of \autoref{th:herschszego}(a), except for the following modification due to the measure being signed. When estimating from below the component of $V(x)$ in direction $z$ we must  take the absolute value of the measure of the point mass:
\[
\liminf_{x \to z} z \cdot V(x) \geq \mu(\overline{\Bn} \setminus \{ -z \}) - |\mu(\{ -z \})| = 2\e(z) ,
\]
where
\[
\e(z) = \frac{1}{2} \mu(\overline{\Bn}) - \mu( \{ -z \} )^+ .
\]
If $\mu( \{ -z \} ) \leq 0$ then $\e(z)>0$, since $\mu(\overline{\Bn})>0$ by hypothesis. If $\mu( \{ -z \} ) > 0$ then $\e(z)>0$ by the hypothesis in part (a). Either way we get $\e(z)>0$, and so the proof can be completed as for \autoref{th:herschszego}(a). 

\smallskip Part (b). Apply part (a) with $g(r)=r$ and $\mu$ supported on the sphere, and use that $|T_x(y)|=1$ whenever $|y|=1$ and so $v(T_x(y))=T_x(y)$.

\smallskip Part (c). After replacing $f$ with $-f$ if necessary, we may suppose $\int_{S^{n-1}} f \, dS > 0$. Apply part (b) with $d\mu = f \, dS|_{S^{n-1}}$, noting that each point on the sphere $S^{n-1}$ has $\mu$-measure $0$, since $n \geq 2$. 

\smallskip Part (d). After replacing $f$ with $-f$ if necessary, we may suppose $\int_\Omega f(z) \, |dz|^2 > 0$. Apply part (a) with $n=2$ and with $d\mu$ being the pushforward of $f(z) \, |dz|^2$ under $F$, that is, $\mu(A) = \int_{F^{-1}(A)} f(z) \, |dz|^2$ for Borel sets $A \subset \D$, and $\mu=0$ on the unit circle. 

\smallskip Part (e). After replacing $f$ with $-f$ if necessary, we may suppose $\int_J f(z) \, |dz| > 0$. Apply part (b) with $n=2$ and with $d\mu$ being the pushforward of $f \, |dz|$ under $F$, that is, $\mu(A) = \int_{F^{-1}(A)} f \, |dz|$ for Borel sets $A \subset \partial \D$. Note each point on the circle has $\mu$-measure $0$, since its inverse image under $F$ is a single point on the Jordan curve.

\section{\bf Differential geometric formulation of the center of mass and energy, and relation to the Riemannian (Karcher) center of mass}  \label{sec:diffgeometric}	

This paper takes advantage of the Poincar\'{e} ball model of hyperbolic space, and generally employs the euclidean radius $r$ rather than the hyperbolic radius $s=s(r)=\arctanh r$. Formulas are stated in terms of the euclidean length $|T_x(y)|$ rather than the hyperbolic distance from $-x$ to $y$. To express the center of mass and energy in a more intrinsic, differential geometric fashion, such as used in \cite[Lemma 4.11]{ABC09} and elsewhere, one may proceed as follows for a measure $\mu$ supported in the open ball $\Bn$. 

\subsection*{Geometric formulation of the center of mass}
We claim that the vector field 
\[
V(-x) = \int_\Bn g(|T_{-x}(y)|) \frac{T_{-x}(y)}{|T_{-x}(y)|} \, d\mu(y) ,
\] 
whose vanishing determines the center of mass point, can be rewritten as 
\begin{equation} \label{eq:diffgeom}
V(-x) = \int_\Bn \widetilde{g}(d_\Bn(x,y)) \frac{\exp_x^{-1}(y)}{d_\Bn(x,y)} \, d\mu(y) 
\end{equation}
where $\exp_x$ is the exponential map from the tangent space at $x$ into $\Bn$, and $\widetilde{g}(s)=g(r)$. 

\noindent \emph{Proof of \eqref{eq:diffgeom}.} Observe first that 
\begin{equation} \label{eq:dseq}
d_\Bn(x,y) = d_\Bn(0,T_{-x}(y)) = s(|T_{-x}(y)|) ,
\end{equation}
since $T_{-x}$ is a hyperbolic isometry that maps $x$ and $y$ to the points $0$ and $T_{-x}(y)$, respectively. Hence $\widetilde{g}(d_\Bn(x,y)) = g(|T_{-x}(y)|)$. Next, the exponential map preserves distances and so $|\exp_x^{-1}(y)|=d_\Bn(x,y)$, which means the vector $\exp_x^{-1}(y)$ in \eqref{eq:diffgeom} becomes a unit vector after dividing by the distance. 

To finish proving \eqref{eq:diffgeom}, we need only show the vectors $T_{-x}(y)$ and $\exp_x^{-1}(y)$ point in the same direction. That is, we want the vector $T_{-x}(y)$ to be tangential at $x$ to the geodesic going from $x$ to $y$. That geodesic has parameterization $t \mapsto T_x(tT_{-x}(y))$, with $t=0$ giving $x$ and $t=1$ giving $y$. The tangent vector is
\[
\frac{d\ }{dt} T_x(tT_{-x}(y)) \Big|_{t=0} = (1-|x|^2) T_{-x}(y) .
\]
Thus $T_{-x}(y)$ is tangential at $x$ to the geodesic from $x$ to $y$, as we wanted to show. \qed

\subsection*{Geometric formulation of the energy}
Recall the function $\widetilde{G}(s)=G(r)$. Observe 
\[
\widetilde{G} \big( d_\Bn(x,y) \big) = G(|T_{-x}(y)|) = \Gamma(T_{-x}(y)) 
\]
by \eqref{eq:dseq}, and so the energy functional can be rewritten in terms of the distance function as 
\[
E(-x) = \int_\Bn \widetilde{G}\big( d_\Bn(x,y) \big) \, d\mu(y) , \qquad x \in \Bn .
\]

\subsection*{Riemannian center of mass}
A good account of the Riemannian center of mass as developed by Cartan, Grove, Karcher and others can be found in Jost's book \cite[Chapter 6]{J17} for a nonpositively curved, complete, simply connected manifold $M$ with distance function $d(x,y)$. For comments on the history, and extensions to other manifolds, see Karcher \cite{K14} and Afsari \cite{A11}. The highlights for us are that if the energy
\[
E_2(x) = \int_M \frac{1}{2} d(x,y)^2 \, d\mu(y)
\]
is finite, then it is a strictly convex function of $x \in M$ along each geodesic \cite[Lemma 6.9.5]{J17}, and hence the energy has a unique minimizing point, which satisfies a center of mass equation \cite[Theorem 6.9.4]{J17}. These results build on convexity of the distance function \cite[Corollary 6.9.2]{J17}. 

The Riemannian energy corresponds to the quadratic choice $\widetilde{G}(s)=\frac{1}{2} s^2$, giving $E_2(x)=E(-x)$. For the hyperbolic ball, the corresponding weight is $g(r)=(1-r^2)G^\prime(r)=\widetilde{G}^\prime(s)=s=\arctanh r$. Thus for this choice of strictly increasing weight function, the Riemannian center of mass theorem yields existence and uniqueness of the center of mass in \autoref{th:hyperbolic}(a)(b), for finite measures on the open ball. Afsari \cite[Theorem 2.1]{A11} has treated the $L^p$ energy $E_p(x)$, for which $g(r)=(\arctanh r)^{p-1}$. In the constant-weight case $p=1$, he noted the possible nonuniqueness of the center of mass for measures supported in a geodesic. These authors presumably recognized that existence and uniqueness continue to hold for any strictly increasing weight $g(r)$. 

\section*{Acknowledgments}
This research was supported by a grant from the Simons Foundation (\#429422 to Richard Laugesen) and the University of Illinois Research Board (RB19045). I am grateful to  Alexandre Girouard and Jeffrey Langford for stimulating conversations about center of mass results. Mark Ashbaugh generously provided considerable assistance with understanding the literature.

\appendix

\section{\bf Convexity of the hyperbolic distance from the origin}  \label{sec:geometricfact}	

The following fact was needed in the proof of \autoref{th:hyperbolic}(b), for uniqueness
\begin{lemma} \label{le:hypgeo}
The hyperbolic distance from the origin, which is 
\[
d_\Bn(x,0) = s(|x|) = \frac{1}{2} \log \frac{1+|x|}{1-|x|} , 
\]
is a hyperbolically convex function of $x \in \Bn$.

In fact, $d_\Bn(x,0)$ is strictly hyperbolically convex along each geodesic that does not pass through the origin. On the straight line geodesics passing through the origin, $d_\Bn(x,0)$ is hyperbolically linear on each side of the origin, with its graph having a corner at the origin.
\end{lemma}
Geometrically elegant proofs are presented by Thurston \cite[Theorem 2.5.8]{T97} and Martelli \cite[Proposition 2.4.4]{M16}. For a generalization to CAT(0) spaces, see Bridson and Haefliger \cite[Chapter II.2]{BH99}. An analytic proof of the lemma is given below. 

The distance from any other point in the ball is hyperbolically convex too, as one deduces by applying an isometry. That is, for each $y \in \Bn$, the map $x \mapsto d_\Bn(x,y)$ is hyperbolically convex.   

The euclidean analogue of the lemma says the distance $|x|$ from the origin is strictly convex along each straight line that avoids the origin, and on straight lines that do pass through the origin, $|x|$ is linear on each side of its corner point at the origin. 
\begin{proof}[Proof of \autoref{le:hypgeo}]
The result is immediate when the geodesic is a straight line through the origin with direction vector $z$, since $d_\Bn(tz,0)=s(|t|)$ for $t \in (-1,1)$, and this function has a corner in its graph at $t=0$. 

Now suppose the geodesic $\gamma$ does not pass through the origin. Consider an arbitrary point $x \in \gamma$, so that $0<|x|<1$. We will show the second derivative of $d_\Bn(\cdot,0)$ along $\gamma$ is positive at $x$. 

The geodesic can be parameterized as $t \mapsto T_x(tz)$ for some unit vector $z$, since the isometry $T_x$ maps geodesics through the origin ($t \mapsto tz$) to geodesics through $x$. Notice $z \neq \pm x/|x|$, since $\gamma$ is not a straight line through the origin. The definition of $T_x$ in \eqref{eq:Mobius} implies that 
\[
T_x(tz) 
= x + t (1-|x|^2) z + t^2 (1-|x|^2) (x - 2 (x \cdot z)z) + O(t^3)
\]
as $t \to 0$. A straightforward calculation shows that 
\[
\begin{split}
& \frac{1+|T_x(tz)|}{1-|T_x(tz)|} \\
& = \frac{1+|x|}{1-|x|} \left( 1 + 2t(\hat{x} \cdot z) + 2t^2 \left[ \frac{1+|x|^2}{2|x|}(1-(\hat{x} \cdot z)^2) + (\hat{x} \cdot z)^2 \right] + O(t^3) \right) ,
\end{split}
\]
where $\hat{x}=x/|x|$. Taking the logarithm of each side, one finds
\[
d_\Bn(T_x(tz),0) = d_\Bn(x,0) + t(\hat{x} \cdot z) + t^2 \frac{1+|x|^2}{2|x|} \left( 1 - (\hat{x} \cdot z)^2 \right) + O(t^3) .
\]
Because $z \neq \pm \hat{x}$ we know $|\hat{x} \cdot z|<1$. Hence the second derivative is positive at $t=0$, 
\[
\left. \frac{d^2\ }{dt^2} \, d_\Bn(T_x(tz),0) \right|_{t=0} > 0 ,
\]
which is equivalent to 
\[
\left. (1-t^2) \frac{d\ }{dt} \left( (1-t^2) \frac{d\ }{dt} \, d_\Bn(T_x(tz),0) \right) \right|_{t=0} > 0 .
\]
That is, $d_\Bn(T_x(\cdot),0)$ is strictly convex at the origin along the straight line geodesic in direction $z$, and so $d_\Bn(\cdot,0)$ is strictly hyperbolically convex along $\gamma$ at $x$.
\end{proof}

\bibliographystyle{plain}

\end{document}